\documentclass{emsprocart}

\usepackage{srcltx}

\contact[reineke@math.uni-wuppertal.de]{Markus Reineke, Fachbereich C - Mathematik, Bergische Universit\"at Wuppertal, Gau\ss str. 20, D - 42097 Wuppertal, Germany}

\newtheorem{theorem}{Theorem}[section]
\newtheorem{corollary}[theorem]{Corollary}
\newtheorem{lemma}[theorem]{Lemma}
\newtheorem{proposition}[theorem]{Proposition}
\newtheorem{conjecture}[theorem]{Conjecture}
\newtheorem{problem}[theorem]{Problem}

\theoremstyle{definition}
\newtheorem{definition}[theorem]{Definition}
\newtheorem{remark}[theorem]{Remark}
\newtheorem{example}[theorem]{Example}

\newcommand{\dimv}{\underline{\dim}}

\title[Moduli of Representations of Quivers]{Moduli of Representations of Quivers}

\author[M. Reineke]{Markus Reineke}

\begin{document}

\begin{abstract}
An introduction to moduli spaces of representations of quivers is given, and results on their global geometric properties are surveyed. In particular, the geometric approach to the problem of classification of quiver representations is motivated, and the construction of moduli spaces is reviewed. Topological, arithmetic and algebraic methods for the study of moduli spaces are discussed.
\end{abstract}



\maketitle

\section{Introduction}\label{introduction}

One of the fundamental problems in the representation theory of algebras is the classification up to isomorphism of the finite-dimensional representations of a given algebra. But ``most'' algebras are wild, in the sense that the problem of classification of their representations is as difficult as the classification of representations of free algebras, or of any wild quiver. This is sometimes referred to as a ``hopeless'', or ``impossible'' problem (see however e.g. \cite{GDP,MFK} for more optimistic opinions). The main obstacle in this case is the dependence of the isomorphism classes of representations on arbitrarily many continuous parameters, to which many of the classical tools of the representation theory of algebras do not apply.\\[1ex]
Nevertheless, one can approach the classification problem geometrically, by ``materializing'' the continuous parametrization phenomena
in spaces whose points correspond naturally to isomorphism classes of representations -- these are the moduli spaces we consider in the present paper. We will focus on global qualitative geometric properties of such spaces, to hopefully derive a qualitative understanding of the classification problem.\\[1ex]
The first aim of this paper is to motivate the geometric approach to the classification problem, the definition of the moduli spaces, and in particular the notion of stability of representations. To this end, we first collect some very basic observations and examples in section \ref{theclass}. We recall the basic concepts of Geometric Invariant Theory as in \cite{Mukai}, and use them to construct the moduli spaces of stable representations, following \cite{King}, in section \ref{defmoduli}. In section \ref{aas}, we discuss the purely algebraic aspects of the notion of stability, following \cite{F,HN,Ru}. In this latter section, complete proofs are given.\\[1ex]
Our second aim is to discuss the kind of questions one can ask about the moduli spaces. We first collect several of their geometric properties, mainly following work of H. Derksen, A. King, L. Le Bruyn, A, Schofield, M. Van den Bergh, J. Weyman and others, in section \ref{further}. The notion (and utility) of cell decompositions is discussed in section \ref{sectbettigross}, as a motivation for the topological and arithmetic considerations in all following sections.\\[1ex]
Based on this, our third main theme is the discussion of the methods and results of the author's (and others) work \cite{ERSM,MR,RHNS,RUse,RFQM,RCNH,RCRP,RLoc,RW}. We focus on the (essentially easy) methods which are used in these papers, namely torus localization in section \ref{loc}, counting points over finite fields in section \ref{arithm}, Hall algebras in section \ref{rolehall}, and framings of moduli in section \ref{sm}. As applications of these techniques, the main results of the above papers are stated, together with indications of the proofs. Several conjectures, open questions and potential directions for further research are discussed.

\section{The classification problem for representations of quivers}\label{theclass}

\subsection{Quivers and their representations}

We start by fixing some standard notation for quivers and their representations (for all basic notions of the representation theory of quivers and of finite dimensional algebras, we refer to \cite{ASS,ARS}). Let $Q$ be a finite quiver, possibly with oriented cycles, which is given by a finite set of vertices $I$ and a finite set of arrows $Q_1$. The arrows will be denoted by $(\alpha:i\rightarrow j)\in Q_1$. We denote by ${\bf Z}I$ the free abelian group generated by $I$. An element $d\in{\bf Z}I$ will be written as $d=\sum_{i\in I}d_ii$. On ${\bf Z}I$, we have the Euler form of $Q$, a non-symmetric bilinear form, defined by
$$\langle d,e\rangle=\sum_{i\in I}d_ie_i-\sum_{(\alpha:i\rightarrow j)\in Q_1}d_ie_j$$
for $d,e\in{\bf Z}I$. We denote by $\dim d=\sum_{i\in I}d_i$ the total dimension of $d$. Standard examples of quivers in this paper will be:
\begin{itemize}
\item the $m$-loop quiver $L_m$ with $I=\{i\}$ and $Q_1=\{(\alpha_1,\ldots,\alpha_m:i\rightarrow i)\}$,
\item the $m$-arrow generalized Kronecker quiver $K_m$ with $I=\{i,j\}$ and $Q_1=\{(\alpha_1,\ldots,\alpha_m:i\rightarrow j)\}$,
\item the $m$-subspace quiver $S_m$ with $I=\{i_1,\ldots,i_m,j\}$ and $Q_1=\{(\alpha_k:i_k\rightarrow j)\, :\, k=1,\ldots,m\}$.
\end{itemize}

We fix an algebraically closed field of characteristic $0$. Let ${\rm mod}\, kQ$ be the abelian category of finite-dimensional representations of $Q$ over $k$ (or, equivalently, finite dimensional representations of the path algebra $kQ$). Its objects are thus given by tuples $$M=((M_i)_{i\in I},(M_\alpha:M_i\rightarrow M_j)_{\alpha:i\rightarrow j})$$ of finite dimensional $k$-vector spaces and $k$-linear maps between them. The dimension vector $\dimv M\in{\bf N}I$ is defined as $\dimv M=\sum_{i\in I}\dim_kM_ii$. We denote by ${\rm Hom}(M,N)$ (resp.~${\rm Ext}^1(M,N)$) the $k$-vector space of homomorphisms (resp. extension classes) between two representations $M,N\in{\rm mod}\, kQ$ (all higher ${\rm Ext}$'s vanish since the category ${\rm mod}\, kQ$ is hereditary).

\subsection{The classification problem}

The basic problem in the representation theory of $Q$ is the following:

\begin{problem}\label{basicproblem} Classify the representations in ${\rm mod}\, kQ$ up to isomorphism, and give normal forms for the representations.
\end{problem}

This problem is solved in the case of Dynkin quivers and extended Dynkin quivers (i.e.~the unoriented graph underlying $Q$ is a disjoint union of Dynkin graphs of type $A_n$, $D_n$, $E_6$, $E_7$ or $E_8$ or their extended versions $\widetilde{A}_n$, $\widetilde{D}_n$, $\widetilde{E}_6$, $\widetilde{E}_7$ or $\widetilde{E}_8$) in \cite{Gabriel} and \cite{DF,Nazarova}, respectively (see also \cite{ASS,SS1}). By the Krull-Schmidt theorem, the problem is reduced to the classification of indecomposable representations. These are classified by a discrete parameter (their dimension vector, which is a positive root for the associated root system), together with at most one continuous parameter. The precise statement will be given in section \ref{kt} in the form of Kac's theorem.\\[1ex]
For all other quivers, called wild, the classification problem is unsolved \cite[Chapters XVII and XIX]{SS2} (see however \cite{Ker} for the study of discrete aspects of the problem, i.e.~the Auslander-Reiten theory of wild quivers ). The main focus of the present paper lies on approaching the classification problem with geometric methods, and on trying to understand its nature qualitatively.\\[1ex]
Even the very statement of Problem \ref{basicproblem} leaves much room for interpretation: what could be the nature of a solution of the classification problem? By what kind of object, and by how explicit an object, should the representations be classified?\\[1ex]
For example, is it a solution to parametrize a certain class of representations by the points of some algebraic variety which (at least in principle) can be described by explicit defining (in-)equalities in some projective space? This is what moduli spaces of representations can achive (at least for the class of stable representations defined in section \ref{aas}), see Corollary \ref{corstable}.\\[1ex]
And what is a normal form for a representation? It could mean a recipe for producing a representative of an isomorphism class of a representation $M$ (i.e.~describing the matrices $M_\alpha$ representing the arrows $\alpha$ of $Q$) from a certain set of discrete and continuous invariants. But should these invariants determine $M$ uniquely? And might there be different set of invariants producing isomorphic representations? This is unavoidable even in trivial examples, see section \ref{exdiag} below. Or does an algorithm for constructing representations explicitely (like the reflection functors of \cite{BGP} in the case of Dynkin quivers) qualify as a normal form?\\[1ex]
In section \ref{hilbert}, we will see that it is indeed possible -- in a rather direct combinatorial way -- to obtain a classification together with explicit normal forms for a problem closely related to Problem \ref{basicproblem}, namely the classification of representations of quivers together with a fixed presentation as a quotient of a given finitely generated projective representation.


\subsection{Kac's Theorem}\label{kt}

We will now recall Kac's theorem, which shows that for wild quivers, the classification problem for (indecomposable) representations depends on arbitrarily many continuous parameters. This gives an explanation for the essentially different nature of the classification problem for wild quivers. For more details on the theorem and its proof, see \cite{Kac1,Kac15,KRi}.\\[1ex]
Let $Q$ be a quiver without oriented cycles. We denote by $$(d,e)=\langle d,e\rangle-\langle e,d\rangle$$ the symmetrization of the Euler form of $Q$. On ${\bf Z}I$, we define reflections $s_i$ for $i\in I$ by $$s_i(d)=d-(d,i)i,$$ and we define the Weyl group $W(Q)$ as the subgroup of ${\rm GL}({\bf Z}I)$ generated by the $s_i$ for $i\in I$. The fundamental domain $F(Q)$ is defined as the set of all non-zero dimension vectors $0\not=d\in{\bf N}I$ with connected support (i.e.~the full subquiver with set of vertices ${\rm supp}(d)=\{i\in I,\; d_i\not=0\}$ is connected), such that $(d,i)\leq 0$ for all $i\in I$. The set of real roots $$\Delta^{\rm re}(Q)=W(Q)I\subset{\bf Z}I$$ is defined as the set of all Weyl group translates of coordinate vectors, and the set of (positive) imaginary roots $$\Delta^{\rm im}(Q)=W(Q)F(Q)$$ is defined as the set of all Weyl group translates of elements of the fundamental domain. Finally, we define $\Delta(Q)$ as the union of both sets of roots.

\begin{theorem} There exists an indecomposable representation of $Q$ of dimension vector $d\in{\bf N}I$ if and only if $d$ is a root, i.e.~$d\in\Delta(Q)$. In case $d\in\Delta^{\rm re}(Q)$, there exists a unique indecomposable (up to isomorphism) of dimension vector $d$.
In case $d\in\Delta^{\rm im}(Q)$, the number of parameters of the set of indecomposable representations is $1-\langle d,d\rangle$. \end{theorem}

The last part of the theorem roughly means that, for an imaginary root $d$, the isomorphism classes of indecomposables of dimension vector $d$ form a family depending on $1-\langle d,d\rangle$ continuous parameters. We will not need the precise definition of the notion of numbers of parameters in the following (this notion associates to an action of an algebraic group $G$ on a constructible subset $X$ of an algebraic variety a number $p$, which equals the dimension of a quotient variety $X/G$ in case it exists, which is not true in general. In our situations, quotient varietes -- the moduli spaces -- will exist, thus there is no need for this more general concept).

\subsection{Generic classification of representations -- an example}\label{genex}

The following trivial example shows that it is often rather easy to classify representations ``generically'', and that this has consequences for the type of questions that should be asked on the nature of Problem \ref{basicproblem}.\\[1ex] 
Consider the three-arrow Kronecker quiver $K_3$, and consider the dimension vector $d=(n,n)$ for $n\geq 1$. In view of Kac's theorem, we expect the classification of indecomposable representations of dimension vector $d$ to depend on $$1-\langle d,d\rangle=n^2+1$$ continuous parameters. Such a representation is given by a tripel $(A,B,C)$ of $n\times n$-matrices.  Assume that $A=I_n$ is the identity matrix, that $B={\rm diag}(\lambda_1,\ldots,\lambda_n)$ is diagonal with pairwise different diagonal entries, and that $C$ is of the form
$$C=\left[\begin{array}{lllll}a_{1,1}&a_{1,2}&\ldots&a_{1,n-1}&a_{1,n}\\
1&a_{2,2}&\ldots&a_{2,n-1}&a_{2,n}\\
a_{3,1}&1&\ldots&a_{3,n-1}&a_{3,n}\\
&&\ldots&&\\
a_{n,1}&a_{n,2}&\ldots&1&a_{n,n}
\end{array}\right]$$
(thus $C_{k+1,k}=1$ for all $k=1,\ldots,n-1$), and that all $a_{i,j}$ are non-zero. Call this representation $M(\lambda_*,a_{**})$.
\begin{lemma} Almost all representations of $K_3$ of dimension vector $d$ (i.e.~a Zariski-open set in the space $M_n(k)^3$ of all representations) are isomorphic to one of the representations $M(\lambda_*,a_{**})$. All the $M(\lambda_*,a_{**})$ are indecomposable with trivial endomorphism ring. Given a tuple $(\lambda_*,a_{**})$ as above, there are only finitely many (in fact, at most $n!$) such tuples $(\lambda_*',a_{**}')$ such that $M(\lambda_*',a_{**}')\simeq M(\lambda_*,a_{**})$.
\end{lemma}

\begin{proof} Let $M$ be an arbitrary representation of $K_3$ of dimension vector $d$, given by matrices $A,B,C$.
Generically, we can assume the matrix $A$ to be invertible (since this is characterized by the open condition of non-vanishing of its determinant), and we can assume the second matrix to be diagonalizable with pairwise different eigenvalues (since this is characterized by the open condition that its characteristic polynomial has no multiple zeroes). Up to the action of the base change group ${\rm GL}_n(k)\times{\rm GL}_n(k)$, we can thus assume that $A=I_n$, the identity matrix, and that $B={\rm diag}(\lambda_1,\ldots,\lambda_n)$. The subgroup of ${\rm GL}_n(k)\times{\rm GL}_n(k)$ fixing the matrices $A$ and $B$ is the diagonally embedded subgroup $T_n(K)$ of diagonal matrices, acting on $C$ via conjugation. Generically, all entries of $C=(a_{i,j})_{i,j}$ are non-zero, and $T_n(K)$ acts on $C$ by
$${\rm diag}(t_1,\ldots,t_n)*C=(\frac{t_i}{t_j}C_{i,j})_{i,j}.$$
Thus, we can assume without loss of generality that $C_{i+1,i}=1$ for all $i$. In other words, almost all representations $M$ are isomorphic to a representation $M(\lambda_*,a_{**})$ as defined above.\\[1ex]
Now assume that two such representations $M(\lambda_*,a_{**})$ and $M(\lambda_*',a_{**}')$, given by triples of matrices
$$(I_n,{\rm diag}(\lambda_1,\ldots,\lambda_n),C)\mbox{, resp. }(I_n,{\rm diag}(\lambda_1',\ldots,\lambda_n'),C'),$$
are isomorphic. This means that there exists a matrix $g\in{\rm GL}_n(k)$ such that
$$g\, {\rm diag}(\lambda_1,\ldots,\lambda_n)={\rm diag}(\lambda_1',\ldots,\lambda_n')\, g\mbox{ and }gC=C'g.$$
The first condition implies that $g$ is a monomial matrix, i.e.~there exists a permutation $\sigma\in S_n$ such that $g_{i,j}=0$ for $j\not=\sigma(i)$, and that $\lambda_i'=\lambda_{\sigma(i)}$. Setting $g_{i,\sigma(i)}=p_i$, the condition $gCg^{-1}=C'$ reads
$$\frac{p_i}{p_j}a_{\sigma(i),\sigma(j)}=a_{ij}'$$
for all $i,j=1,\ldots,n$. In particular, this holds for the pairs $(i+1,i)$, thus the $p_i$ are given inductively as
$$p_k=(\prod_{l<k}a_{\sigma(l+1),\sigma(l)})^{-1}p_1.$$
This shows that at most $n!$ (depending on the permutation $\sigma$) such matrices $C'$ are conjugate to $C$ under $g$.\\[1ex]
Finally, consider the case that $\lambda_*'=\lambda_*$ and $a_{**}'=a_{**}$ componentwise. The above computation shows that $\sigma$ is the identity, and that all $p_i$ are determined by $p_1$. Thus, the endomorphism ring of $M_(\lambda_*,a_{**})$ reduces to the scalars.\end{proof}

This elementary example is instructional for several reasons. First of all, we see that it is often very easy to classify almost all indecomposables of a given dimension vector -- in fact, the above example can be easily generalized to other dimension vectors for a generalized Kronecker quiver, or to other quivers. But such a generic classification does not capture the ``boundary phenomena'' which constitute the essence of the classification problem. In a very coarse analogy, one could argue that for the one-loop quiver $L_1$, a generic classification reduces to the classification of diagonalizable matrices by their eigenvalues, completely missing the Jordan canonical form (i.e.~all higher-dimensional indecomposable representations).\\[1ex]
Second, in the example we can see that there is no easy way to refine the given matrices to a normal form, in the sense that different systems of parameters give non-isomorphic representations. This phenomenom also can be seen in an even more simple example: \begin{example}\label{exdiag} In parametrizing semisimple representations of the one-loop quiver up to isomorphism, i.e.~diagonalizable matrices up to conjugation, the most natural normal form is ${\rm diag}(\lambda_1,\ldots,\lambda_n)$ for $\lambda_1,\ldots,\lambda_n\in{k}$, but this is only unique up to permutation. On the other hand, we have a classification up to isomorphism by the characteristic polynomial, but this does not yield a normal form.
\end{example}

This phenomenom of ``normal forms up to a finite number of identifications'' is in fact closely related to the rationality problem for moduli spaces of quivers, to be discussed in section \ref{ned}.

\section{Definition of moduli spaces and basic geometric properties}\label{defmoduli}

\subsection{Geometric approach}

Our basic approach to the continuous parametrization phenomena in the classification of quiver representations is to ``materialize'' the inherent continuous parameters, by defining geometric objects parametrizing isomorphism classes of certain types of representations (or, as we will consider in section \ref{sm}, representations together with some appropriate additional structure). See \cite{Bo,G,Kraft,RUse} for overviews over the use of geometric techniques in the representation theory of algebras.\\[1ex]
One basic decision we have to make a priori is in which geometric category we want to work. For example, it is possible almost tautologically to define stacks \cite{LMB} parametrizing isomorphism classes of arbitrary representations. But the geometric intuition seems to be lost almost completely for these objects.\\[1ex]
Therefore, we will try to define such parameter spaces in the context of algebraic varieties. Thus, we will use in the following the language of varieties, the Zariski topology on them, etc..\\[1ex]
The basic idea behind this geometric approach is very simple:
Fix a dimension vector $d$, and fix $k$-vector spaces $M_i$ of dimension $d_i$ for all $i\in I$.
Consider the affine $k$-space $$R_d=R_d(Q)=\bigoplus_{\alpha:i\rightarrow j}{\rm Hom}_k(M_i,M_j).$$
Its points $M=(M_\alpha)_\alpha$ obviously parametrize representations of $Q$ on the vector spaces $M_i$. The reductive linear algebraic group $$G_d=\prod_{i\in I}{\rm GL}(M_i)$$ acts on $R_d$ via the base change action
$$(g_i)_i\cdot(M_\alpha)_\alpha=(g_jM_\alpha g_i^{-1})_{\alpha:i\rightarrow j}.$$
By definition, the $G_d$-orbits $G_d*M$ in $R_d$ correspond bijectively to the isomorphism classes $[M]$ of $k$-representations of $Q$ of dimension vector $d$. Our problem of defining parameter spaces can therefore be rephrased as follows:
\begin{problem}\label{geometricapproach} Find a subset $U\subset R_d$, an algebraic variety $X$ and a morphism $\pi:U\rightarrow X$, such that the fibres of $\pi$ are precisely the orbits of $G_d$ in $U$.
\end{problem}

More precisely, we ask the subset $U$ to be ``as large as possible'' (and in particular to be a Zariski open subset of $R_d$), to capture as much of the ``boundary phenomena'' as possible, as motivated by section \ref{genex}.

\subsection{Indecomposables and geometry -- an example}\label{ex5k}

We will now exhibit an example that, in general (even for quivers without oriented cycles and indivisible dimension vectors), a set $U$ as in Problem \ref{geometricapproach} cannot be the set of indecomposable representations, or even the set of Schur representations.\\[1ex]
Consider the $5$-arrow Kronecker quiver $K_5$ and the dimension vector $d=(2,3)$. Define a family of representations $M(\lambda,\mu)$ for $(\lambda,\mu)\not=(0,0)$ by the five matrices
$$\left[\begin{array}{ll}1&0\\0&0\\0&0\end{array}\right],\;
\left[\begin{array}{ll}0&0\\0&1\\0&0\end{array}\right],\;
\left[\begin{array}{ll}0&0\\0&0\\0&1\end{array}\right],\;
\left[\begin{array}{ll}0&0\\ \lambda&0\\0&0\end{array}\right],\;
\left[\begin{array}{ll}0&\mu\\0&0\\0&0\end{array}\right].$$

The following lemma is proved by a direct calculation.

\begin{lemma} All $M(\lambda,\mu)$ for $(\lambda,\mu)\not=(0,0)$ are Schur representations. We have $M(\lambda,\mu)\simeq M(\lambda',\mu')$ if and only if $\lambda'=t\lambda$ and $\mu'=\frac{1}{t}\mu$ for some $t\in k^*$.
\end{lemma}

Now suppose there exists a subset $U$ of $R_{(2,3)}(Q)$ which contains all $M(\lambda,\mu)$ for $(\lambda,\mu)\not=(0,0)$, together with a variety $X$ and a morphism $\pi:U\rightarrow X$ as above.
We have
$$\lim_{t\rightarrow 0}M(t,1)=M(0,1)\mbox{ and }\lim_{t\rightarrow 0}M(1,t)=M(1,0)\mbox{ in }U.$$
Applying the continuous map $\pi$, we get a contradiction: we have $$M(t,1)\simeq M(1,t)\mbox{ for all }t\not=0\mbox{, but }M(0,1)\not\simeq M(1,0).$$.

Thus, in trying to define a set $U$ as in Problem \ref{geometricapproach}, we have to make a choice between $M(0,1)$ and $M(1,0)$. It will turn out that these two representations can be distinguished by the structure of their submodule lattices. For example, $M(0,1)$ admits a subrepresentation of dimension vector $(1,1)$, whereas $M(1,0)$ does not.\\[1ex]
The above example is just one of the typical problems in defining quotients: it shows that the potential ``quotient variety $U/G_d$'' would be a non-separated scheme. In fact, even more severe problems are encountered, for example situations with ``nice'' actions of the group $G_d$ on an open subset $U$ of $R_d$, where nevertheless the desired $X$ cannot be realized as a quasiprojective variety, see \cite{Bo}.\\[1ex]
The above example, most importantly, tells us that the class of indecomposable (or Schurian) representations is not well-adapted to our geometric approach to the classification problem!

\subsection{Review of Geometric Invariant Theory}\label{git}

Geometric Invariant Theory provides a general way of defining subsets $U$ as in Problem \ref{geometricapproach}. We follow mainly the presentation of \cite{Mukai}, since the generality of the standard text \cite{MFK} is not used in the following.\\[1ex]
We will first consider the general problem of constructing quotient varieties which parametrize orbits of a group acting linearly on a vector space.\\[1ex]
Suppose we are given a vector space $V$ together with a linear representation of a reductive algebraic group $G$ on $V$.
A regular (that is, polynomial) function $f:V\rightarrow k$ on $V$ is called an invariant if $$f(gv)=f(v)\mbox{ for all }g\in G\mbox{ and }v\in V.$$
We denote by $k[V]^G$ the subring of invariant functions of the ring $k[V]$ of all regular functions on $V$.\\[1ex]
Since the group $G$ is reductive, a theorem of Hilbert asserts that the invariant ring $k[V]^G$ is finitely generated, thus it qualifies as the coordinate ring of a variety $$V//G:={\rm Spec}(k[V]^G).$$
The embedding of $k[V]^G$ in $k[V]$ dualizes to a morphism (associating to an element $v\in V$ the ideal of invariants vanishing on $v$) denoted by $\pi:V\rightarrow V//G$.\\[1ex] This morphism is $G$-invariant by definition, and it fulfills the following universal property:\\[1ex]
Given a $G$-invariant morphism $h:V\rightarrow X$, there exists a unique map $\overline{h}:V//G\rightarrow X$ such that $h=\overline{h}\circ \pi$.\\[1ex]
This property shows that $\pi$ can be viewed as the optimal approximation (in the category of algebraic varieties) to a quotient of $V$ by $G$.\\[1ex]
Moreover, one can prove that any fibre of $\pi$ (which is necessarily $G$-stable) contains exactly one closed $G$-orbit.
The quotient variety $X$ therefore parametrizes the closed orbits of $V$ in $G$.\\[1ex]
To obtain an open subset $U$ as in Problem \ref{geometricapproach}, we define $V^{\rm st}$ as the set of all points $v\in V$ such that the orbit $Gv$ is closed, and such that the stabilizer of $v$ in $G$ is zero-dimensional (thus, finite). Then we have the following:

\begin{theorem}\label{vst} $V^{\rm st}$ is an open (but possibly empty) subset of $V$, and the restriction $$\pi:V^{\rm st}\rightarrow \pi(V^{\rm st})=:V^{\rm st}/G$$ gives a morphism whose fibres are exactly the $G$-orbits in $V^{\rm st}$.
\end{theorem}

We consider now a relative version of this construction:
Choose a character of $G$, that is, a morphism of algebraic groups $\chi:G\rightarrow k^*$.  A regular function $f$ is called $\chi$-semi-invariant if $$f(gv)=\chi(g)f(v)\mbox{ for all } g\in G\mbox{ and }v\in V.$$
We denote by $k[V]^{G,\chi}$ the subspace of $\chi$-semi-invariants, and by
$$k[V]^G_\chi:=\bigoplus_{n\geq 0}k[V]^{G,\chi^n}$$
the subring of semi-invariants for all powers of the character $\chi$. This is naturally an ${\bf N}$-graded subring of $k[V]$ with $k[V]^G$ as the subring of degree $0$ elements.\\[1ex]
An element $v\in V$ is called $\chi$-semistable if there exists a function $f\in k[V]^{G,\chi^n}$ for some $n\geq 1$ such that $f(v)\not=0$. Denote by $V^{\chi-{\rm sst}}$ the (open) subset of $\chi$-semistable points. An element $v$ is called $\chi$-stable if the following conditions are satisfied: $v$ is $\chi$-semistable, its orbit $Gv$ is closed in $V^{\chi-{\rm sst}}$, and its stabilizer in $G$ is zero-dimensional. Denote by $V^{\chi-{\rm st}}$ the (again open) subset of stable points.\\[1ex]
Let $V^{\chi-{\rm sst}}//G$ be the quasiprojective variety defined as ${\rm Proj}(k[V]^G_\chi)$. Associating to $v\in V^{\chi-{\rm sst}}$ the ideal of functions $f$ vanishing on $v$ gives -- by definition of semistability -- a well-defined map $\pi$ from $V^{\rm \chi-sst}$ to $V^{\chi-{\rm sst}}//G$.\\[1ex]
Then the following properties, similar to Theorem \ref{vst}, hold:

\begin{theorem} The variety $V^{\chi-{\rm sst}}//G$ parametrizes the closed orbits of $G$ in $V^{\chi-{\rm sst}}$, and the restriction of $\pi$ to $V^{\chi-{\rm st}}$ has as fibres precisely the $G$-orbits in $V^{\chi-{\rm st}}$.
\end{theorem}

Therefore, we see that the open subsets of $\chi$-stable points in $V$ with respect to a given character $\chi$ of $G$ give us candidates for open subsets as in Problem \ref{geometricapproach}. The case of Theorem \ref{vst} corresponds to the trivial choice $\chi={\rm id}$ of character - in this case, all points are semistable.\\[1ex]
We summarize the basic geometric properties of quotients in the following diagram:
$$\begin{array}{ccccc}V^{\chi-{\rm st}}&\subset&V^{\chi-{\rm sst}}&\subset&V\\
\downarrow&&\downarrow&&\downarrow\\
V^{\chi-{\rm st}}/G&\subset&V^{\chi-{\rm sst}}//G&\rightarrow&V//G\\
&&||&&||\\
&&{\rm Proj}(k[V]^G_\chi)&\rightarrow&{\rm Spec}(k[V]^G)\end{array}$$
\begin{proposition}\label{collprop} The following properties hold:
\begin{itemize}
\item $V^{\chi-{\rm st}}\subset V^{\chi-{\rm sst}}\subset V$ is a chain of open inclusions,
\item the variety $V//G$ is affine,
\item the morphism $V^{\chi-{\rm sst}}//G\rightarrow V//G$ is projective,
\item if the action of $G$ on $V^{\chi-{\rm st}}$ is free, the morphism $V^{\chi-{\rm st}}\rightarrow V^{\chi-{\rm st}}/G$ is a $G$-principal bundle,
\item if $V^{\chi-{\rm st}}/G$ is non-empty, its dimension equals $\dim V-\dim G$.
\end{itemize}
\end{proposition}

\subsection{Geometric Invariant Theory for quiver representations}\label{gitquiver}

We will now apply the above constructions to the action of the group $G_d$ on the vector space $R_d(Q)$; the details can be found in \cite{King}.\\[1ex]
First of all, we note that the (diagonally embedded) scalar matrices in $G_d$ act trivially on $R_d(Q)$ (and therefore are contained in the stabilizer of any point). This means that we have to pass to the action of the factor group $PG_d=G_d/k^*$ first to admit orbits with zero-dimensional stabilizers.\\[1ex]
An orbit $G_d*M=PG_d*M$ is closed in $R_d(Q)$ if and only if the corresponding representation $M$ is semisimple by \cite{Artin}. The quotient variety $R_d(Q)//PG_d$ therefore parametrizes isomorphism classes of semisimple representations of $Q$ of dimension vector $d$. It will be denoted by $M_d^{\rm ssimp}(Q)$ and called the moduli space of semisimple representations.\\[1ex]
The stabilizer of $M$ in $G_d$ is nothing else than the automorphism group ${\rm Aut}(M)$, thus its stabilizer in $PG_d$ is zero-dimensional (in fact trivial) if and only if $M$ is a Schur representation. Thus, we see that the open subset $R_d(Q)^{\rm st}$ consists precisely of the simple representations of dimension vector $d$, and that the quotient variety $R_d(Q)^{\rm st}(Q)/PG_d=M_d^{\rm simp}(Q)$ is a moduli space for (isomorphism classes of) simple representations.\\[1ex]
This shows that the concept of quotients in its original form gives interesting moduli spaces only in the case of quivers with oriented cycles, since for a quiver without oriented cycles, the only simple representations are the one-dimensional ones $S_i$ of dimension vector $i$ for $i\in I$. Therefore, the generality of the relative setup is necessary to obtain interesting moduli spaces.\\[1ex]
The characters of the general linear group are just integer powers of the determinant map, thus the characters of the group $PG_d$ are of the form
$$(g_i)_i\mapsto\prod_{i\in I}\det(g_i)^{m_i},$$
for a tuple $(m_i)_{i\in I}$ such that $\sum_{i\in I}m_id_i=0$ to guarantee well-definedness on $PG_d$.\\[1ex]
Thus, let us choose a linear function $\Theta:{\bf Z}I\rightarrow{\bf Z}$ and associate to it a character
$$\chi_\Theta((g_i)_i):=\prod_{i\in I}\det(g_i)^{\Theta(d)-\dim d\cdot \Theta_i}$$
of $PG_d$ (this adjustment of $\Theta$ by a suitable multiple of the function $\dim$ has the advantage that a fixed $\Theta$ can be used to formulate stability for arbitrary dimension vectors, and not only those with $\Theta(d)=0$). We will denote the corresponding sets of $\chi_\Theta$-(semi-)stable points by
$$R_d^{\rm sst}(Q)=R_d^{\rm \Theta-sst}(Q)=R_d(Q)^{\rm \chi_\Theta-sst}\mbox{ and }R_d^{\rm st}(Q)=R_d^{\rm \Theta-st}(Q)=R_d(Q)^{\rm \chi_\Theta-st}.$$
The corresponding quotient varieties will be denoted as follows:
$$M_d^{\rm (\Theta-)st}(Q)=R_d^{\rm \Theta-st}(Q)/G_d\mbox{ and }M_d^{\rm (\Theta-)sst}(Q)=R_d^{\rm \Theta-sst}(Q)//G_d.$$

As an immediate application of the concepts of section \ref{git}, we get:

\begin{corollary}\label{corstable} The variety $M_d^{\rm \Theta-st}(Q)$ parametrizes isomorphism classes of $\Theta$-stable representations of $Q$ of dimension vector $d$.
\end{corollary}

The closed orbits of $G_d$ in $R_d^{\rm \Theta-sst}(Q)$ correspond to the so-called $\Theta$-polystable representations. These are defined as direct sums of stable representations of the same slope. They can also be viewed as the semisimple objects in the abelian subcategory ${\rm mod}_\mu kQ$ N of semistable representations of fixed slope $\mu$ -- see section \ref{aas} for details.
Therefore, we get:

\begin{corollary} 
The variety $M_d^{\rm \Theta-sst}(Q)$ parametrizes isomorphism classes of $\Theta$-polystable representations of $Q$ of dimension vector $d$.
\end{corollary}

\subsection{Basic geometric properties}\label{bgp}

We summarize the varieties appearing in the definition of quiver moduli in the following diagram:

$$\begin{array}{ccccc}R_d^{\rm \Theta-st}(Q)&\subset&R_d^{\rm \Theta-sst}(Q)&\subset&R_d(Q)\\
\downarrow&&\downarrow&&\downarrow\\
M_d^{\rm \Theta-st}(Q)&\subset&M_d^{\rm \Theta-sst}(Q)&\rightarrow&M_d^{\rm ssimp}(Q)\\
\end{array}$$

As an application of Proposition \ref{collprop}, we have the following geometric properties:

\begin{itemize}
\item $R_d^{\rm \Theta-st}(Q)\subset R_d^{\rm \Theta-sst}(Q)\subset R_d(Q)$ is a chain of open inclusions,
\item $M_d^{\rm \Theta-st}(Q)$ is a smooth variety,
\item $M_d^{\rm ssimp}(Q)$ is an affine variety,
\item $M_d^{\rm \Theta-sst}(Q)\rightarrow M_d^{\rm ssimp}(Q)$ is a projective morphism,
\item $R_d^{\rm \Theta-st}(Q)\rightarrow M_d^{\rm \Theta-st}(Q)$ is a $PG_d$-principal fibration.
\item The dimension of the variety $M_d^{\rm \Theta-st}(Q)$, if non-empty, equals $1-\langle d,d\rangle$.
\end{itemize}

The relation between the moduli spaces $M_d^{\rm st}(Q)$ and $M_d^{\rm sst}(Q)$ can be made more precise using the Luna stratification; see \cite{AL} for more details:\\[1ex]
Given a polystable representation $$M=U_1^{m_1}\oplus\ldots\oplus U_s^{m_s}$$ for pairwise non-isomorphic stables $U_k$ of dimension vectors $d^k$ (necessarily of slope $\mu(d)$), we denote its polystable type by $$(d^*)^{m_*}=((d^1)^{m_1}\ldots (d^s)^{m_s}).$$
The subset $S_{d^*}$ of $M_d^{\rm sst}(Q)$ of all polystable representations of type $(d^*)^{m_*}$ is locally closed, yielding a finite stratification of $M_d^{\rm sst}(Q)$, such that the stratum corresponding to type $(d^1)$ is precisely $M_d^{\rm st}(Q)$.\\[1ex]
This can be used, for example, to determine the singularities in $M_d^{\rm sst}(Q)$, see \cite{AL}. Other applications will be mentioned in sections \ref{counting},\ref{dsm}.\\[1ex]
As noted before, in case the quiver $Q$ has no oriented cycles, the only simple representations are the one-dimensional ones $S_i$ attached to each vertex $i\in I$. Thus, in this case, $M_d^{\rm ssimp}(Q)$ consists of a single point, corresponding to the semisimple representation $\bigoplus_{i\in I}S_i^{d_i}$ of dimension vector $d$. Consequently, the variety $M_d^{\rm \Theta-sst}$ is projective in this case.
Now assume that $d$ is coprime for $\Theta$, which by definition means that $\mu(e)\not=\mu(d)$ for all $0\not=e<d$. For generic $\Theta$, this is equivalent to $d$ being coprime in the sense that ${\rm gcd}(d_i\, ,\, i\in I)=1$:\\[1ex]
Suppose that $d$ is $\Theta$-coprime, and that $k\in{\bf N}$ is a common divisor of all entries of $d$. Then $\Theta(\frac{1}{k}d)=\Theta(d)$, and thus $k=1$. Conversely, suppose that $d$ is coprime. Then the conditions $\mu(e)\not=\mu(d)$ for all $0\not=e<d$ define finitely many proper hyperplanes, and a generic choice of $\Theta$ avoids all of them.\\[1ex]
 Then, by assumption, every semistable representation is already stable, thus the variety $M_d^{\rm \Theta-sst}=M_d^{\rm \Theta-st}$ is smooth.\\[1ex]
Thus, in the case of a $\Theta$-coprime dimension vector for a quiver without oriented cycles, we end up which smooth projective moduli spaces, which therefore classify for a consideration with classical techniques of algebraic geometry - see the following sections \ref{sectbettigross}, \ref{loc}, \ref{arithm} for some applications of this principle.\\[1ex]
We have thus found many solutions to Problem \ref{geometricapproach}. The next step is to characterize the stable representations in algebraic terms:\\[1ex]
Define the slope of a non-zero dimension vector as $$\mu(d):=\frac{\Theta(d)}{\dim d},$$
and define the slope of a representation $M\not=0$ as the slope of its dimension vector $\mu(M):=\mu(\dimv M)$. The following characterization of $\chi_\Theta$-(semi-)stable points in $R_d(Q)$ is given in \cite{King}:

\begin{theorem} A representation $M\in R_d(Q)$ is $\chi_\Theta$-semistable if and only if $\mu(U)\leq \mu(M)$ for all non-zero subrepresentations $U$ of $V$. The representation $M$ is $\chi_\Theta$-stable if and only if $\mu(U)<\mu(M)$ for all non-zero proper subrepresentations $U$ of $V$.
\end{theorem}

Thus, (semi-)stability can be characterized in terms of the submodule structure of a representation. Since this simple description of stable representations leads to some very interesting representation-theoretic aspects, we devote the following section to details of this concept.

\section{Algebraic aspects of stability}\label{aas}

As before, let $Q$ be an arbitrary finite quiver, and let $\Theta:{\bf Z}I\rightarrow{\bf Z}$ be a linear function, called a stability in the following. We also consider, as before, the functional $\dim$ on ${\bf Z}I$ defined by $\dim d=\sum_{i\in I}d_i$ (this could be replaced by any other strictly positive linear function on dimension vectors). For a non-zero dimension vector $d\in{\bf N}I$, we define its slope by
$$\mu(d):=\frac{\Theta(d)}{\dim d}\in{\bf Q}.$$
We define the slope of a non-zero representation $X$ of $Q$ over some field $k$ as the slope of its dimension vector, thus $\mu(X):=\mu(\dimv X)\in{\bf Q}$. The set
$${\bf N}I_\mu=\{d\in{\bf N}I\setminus\{0\}\, :\, \mu(d)=\mu\}\cup\{0\}$$
forms a subsemigroup of ${\bf N}I$.\\[1ex]
We call the representation $X$ semistable if $\mu(U)\leq\mu(X)$ for all non-zero subrepresentations $U$ of $X$, and we call $X$ stable if $\mu(U)<\mu(X)$ for all non-zero proper subrepresentations $U$ of $X$.

\begin{lemma}\label{l41} Let $0\rightarrow X\rightarrow Y\rightarrow Z\rightarrow 0$ be a short exact sequence of non-zero $k$-representations of $Q$. Then the following holds:
\begin{enumerate}
\item $\mu(X)\leq\mu(Y)$ if and only if $\mu(X)\leq \mu(Z)$ if and only if $\mu(Y)\leq\mu(Z)$.
\item The same holds with $\leq$ replaced by $<$.
\item $\min(\mu(X),\mu(Z))\leq \mu(Y)\leq\max(\mu(X),\mu(Z))$.
\end{enumerate}
\end{lemma}

\begin{proof} Let $d$ and $e$ be the dimension vectors of $X$ and $Z$, respectively. Then the dimension vector of $Y$ equals $d+e$, and thus the slope of $Y$ equals
$$\mu(Y)=\frac{\Theta(d)+\Theta(e)}{\dim d+\dim e}.$$
It is now trivial to verify that
$$\frac{\Theta(d)}{\dim d}\leq\frac{\Theta(d)+\Theta(e)}{\dim d+\dim e}\iff \frac{\Theta(d)}{\dim d}\leq\frac{\Theta(e)}{\dim e}\iff
\frac{\Theta(d)+\Theta(e)}{\dim d+\dim e}\leq\frac{\Theta(e)}{\dim e},$$
and the same statement with $\leq$ replaced by $<$. This proves the first two parts of the lemma. The third part follows immediately.\end{proof}

This lemma shows that semistability of a representation $X$ can also be characterized by the condition $\mu(X)\leq\mu(W)$ for any non-zero factor representation $W$ of $X$.\\[1ex]
Denote by ${\rm mod}_\mu kQ$ the full subcategory of ${\rm mod}\, kQ$ consisting of semistable representations of slope $\mu\in{\bf Q}$.

\begin{lemma}\label{basicsemistable} The following properties of the subcategories ${\rm mod}_\mu kQ$ hold:
\begin{enumerate}
\item Let $0\rightarrow X\rightarrow Y\rightarrow Z\rightarrow 0$ be a short exact sequence of non-zero $k$-representations of $Q$ of the same slope $\mu$. Then $Y$ is semistable if and only if $X$ and $Z$ are semistable.
\item ${\rm mod}_\mu kQ$ is an abelian subcategory of ${\rm mod}\, kQ$.
\item If $\mu>\nu$, then ${\rm Hom}({\rm mod}_\mu kQ,{\rm mod}_\nu kQ)=0$.
\item The stable representations of slope $\mu$ are precisely the simple objects in the abelian category ${\rm mod}_\mu kQ$. In particular, they are indecomposable, their endomorphism ring is trivial, and there are no non-zero morphisms between non-isomorphic stable representations of the same slope.
\end{enumerate}
\end{lemma}

\begin{proof} Suppose that $X$ and $Z$ are semistable, and let $U$ be a subrepresentation of $Y$. This yields an induced exact sequence
$$0\rightarrow U\cap X\rightarrow U\rightarrow (U+X)/X\rightarrow 0$$
of subrepresentations of $X$, $Y$ and $Z$, respectively. By semistability of $X$ and $Z$, we have $\mu(U\cap X)\leq\mu(X)=\mu$ and $\mu((U+X)/X)\leq\mu(Z)=\mu$. Applying the third part of the previous lemma, we get
$$\mu(U)\leq\max(\mu(U\cap X),\mu((U+X)/X)))\leq\mu=\mu(Y),$$
proving semistability of $Y$.\\[1ex]
Conversely, suppose that $Y$ is semistable. A subrepresentation $U$ of $X$ can then be viewed as a subrepresentation of $Y$, and thus $\mu(U)\leq\mu(Y)=\mu=\mu(X)$, proving semistability of $X$. A subrepresentation $U$ of $Z$ induces an exact sequence $0\rightarrow X\rightarrow V\rightarrow U\rightarrow 0$ by pullback, and thus
$\mu(V)\leq\mu(Y)=\mu=\mu(X)$. Applying the first part of the previous lemma, we get $\mu(U)\leq\mu(V)\leq\mu=\mu(Z)$, proving semistability of $Z$. This proves the first part of the lemma. It also proves that the subcategory ${\rm mod}_\mu kQ$ is closed under extensions.\\[1ex]
Given a morphism $f:X\rightarrow Y$ in ${\rm mod}_\mu kQ$, we have $\mu=\mu(X)\leq\mu({\rm Im}(f))\leq \mu(Y)=\mu$ by semistabililty of $X$ and $Y$, and thus $\mu({\rm Im}(f))=\mu$. Thus, ${\rm Ker}(f)$, ${\rm Im}(f)$ and ${\rm Coker}(f)$ all have the same slope $\mu$, and they are all semistable by the first part. This proves that ${\rm mod}_\mu kQ$ is abelian. The same argument proves the third part: if $f:X\rightarrow Y$ is a non-zero morphism, then $\mu(X)\leq\mu({\rm Im}(f))\leq\mu(Y)$.\\[1ex]
By the definition of stability, a representation is stable of slope $\mu$ if and only if it has no non-zero proper subrepresentation in ${\rm mod}_\mu kQ$, proving that the stables of slope $\mu$ are the simples in ${\rm mod}_\mu kQ$. The remaining statements of the fourth part follow from Schur's Lemma.
\end{proof}

\begin{definition} A subrepresentation $U$ of a representation $X$ is called scss (short for strongly contradicting semistability) if its slope is maximal among the slopes of subrepresentations of $X$, that is, $\mu(U)=\max\{\mu(V)\,|\, V\subset X\}$, and it is of maximal dimension with this property.
\end{definition}

Such a subrepresentation clearly exists, since there are only finitely many dimensions and slopes of subrepresentations. By its defining property, it is clearly semistable.

\begin{lemma} Every representation $X$ admits a unique scss subrepresentation.
\end{lemma}

\begin{proof} Suppose $U$ and $V$ are scss subrepresentations of $X$, necessarily of the same slope $\mu$. The exact sequence $0\rightarrow U\cap V\rightarrow U\oplus V\rightarrow U+V\rightarrow 0$ yields $\mu(U\cap V)\leq\mu=\mu(U\oplus V)$, thus $\mu\leq\mu(U+V)$ by Lemma \ref{l41}. By maximality of the slope $\mu$ among subrepresentations of $X$, we have $\mu(U+V)=\mu$. By maximality of the dimension of $U$ and $V$, we have $\dim U+V\leq \dim U,\dim V$, and thus $U=V$.\end{proof}

\begin{remark}\label{automhnf} The uniqueness of the scss of a representation $X$ has some interesting applications: for example, the scss has to be fixed under arbitrary automorphisms $\varphi$ of $X$, since applying $\varphi$ to a subrepresentations does not change its dimension vector, and thus also its slope and dimension.\\[1ex]
This implies the compatibility of semistability with base extension: let $k\subset K$ be a Galois extension, and let $X$ be a semistable $k$-representation of $Q$. The scss $V$ of $\bar{X}=K\otimes_k X$ is fixed under all automorphism of $\bar{X}$, thus in particular under the Galois group of the extension $k\subset K$. Thus, it descends to a subrepresentation $U$ of $X$ (that is, $V\simeq K\otimes_k U$). By semistability of $X$, we have $\mu(V)=\mu(U)\leq\mu(X)=\mu(\bar{X})$, thus $\bar{X}$ is semistable.\end{remark}

\begin{definition} A filtration $0=X_0\subset X_1\subset\ldots\subset X_s=X$ of a representation $X$ is called Harder-Narasimhan (abbreviated by HN) if the subquotients $X_i/X_{i-1}$ are semistable for $i=1,\ldots,s$, and $\mu(X_1/X_0)>\mu(X_2/X_1)>\ldots>\mu(X_s/X_{s-1})$.
\end{definition}

\begin{lemma} Every representation $X$ possesses a unique Harder-Narasimhan filtration.
\end{lemma}

\begin{proof} Existence is proved by induction on the dimension of $X$. Let $X_1$ be the scss of $X$. By induction, we have a HN filtration $0=Y_0\subset Y_1\subset\ldots \subset Y_{s-1}=X/X_1$. Via the projection $\pi:X\rightarrow X/X_1$, we pull this back to a filtration of $X$ defined by $X_i=\pi^{-1}(Y_{i-1})$ for $i=1,\ldots,s$. Now $X_1/X_0$ is semistable since $X_1$ is the scss of $X$, and $X_{i+1}/X_{i}\simeq Y_i/Y_{i-1}$ is semistable by the choice of the $Y_i$ for $i=1,\ldots,s-1$. We also infer
$\mu(X_2/X_1)>\ldots>\mu(X_s/X_{s-1})$ from the corresponding property of the slopes of the subquotients in the HN filtration of $X/X_1$. Since $X_2$ is a subrepresentation of $X$ of strictly larger dimension than $X_1$, we have $\mu(X_1)>\mu(X_2)$ since $X_1$ is scss in $X$, and thus $\mu(X_1/X_0)=\mu(X_1)>\mu(X_2/X_1)$.\\[1ex]
To prove uniqueness, we proceed again by induction on the dimension. Assume that $0=X_0'\subset\ldots\subset X_s'=X$ is a HN filtration of $X$. Let $t$ be minimal such that $X_1$ is contained in $X_t'$, thus the inclusion induces a non-zero map from $X_1$ to $X_t'/X_{t-1}'$. Both representations being semistable and $X_1$ being scss, we have $\mu(X_1')\leq \mu(X_1)\leq\mu(X_t'/X_{t-1}')\leq\mu(X_1')$, thus $\mu(X_1)=\mu(X_1')$ and $t=1$, which means $X_1\subset X_1'$. Again since $X_1$ is scss, we conclude that $X_1=X_1'$. By induction, we know that the induced filtrations on the factor $X/X_1$ coincide, thus the filtrations of $X$ coincide.\end{proof}

Interpreted properly, the HN filtration is even functorial: index the HN filtration by ${\bf Q}$ by defining
$$X^{(a)}=X_k\mbox{ if }\mu(X_k/X_{k-1})\geq a>\mu(X_{k+1}/X_k)\mbox{ for all }a\in{\bf Q}.$$

\begin{lemma} Any morphism $f:X\rightarrow Y$ respects the HN filtration, in the sense that $f(X^{(a)})\subset Y^{(a)}$ for all $a\in{\bf Q}$.
\end{lemma}

\begin{proof} We will prove by induction on $k$ the following property:\\[1ex]
If $f(X_k)\subset Y_l\setminus Y_{l-1}$, then $\mu(Y_l/Y_{l-1})\geq \mu(X_k/X_{k-1})$.\\[1ex]
The claimed property follows from this: given $a\in{\bf Q}$, we have $X^{(a)}=X_k$ for the index $k$ satisfying $\mu(X_k/X_{k-1})\geq a>\mu(X_{k+1}/X_k)$. Choosing $l$ minimal such that $f(X_k)\subset Y_l$, we then have $\mu(Y_l/Y_{l-1})\geq \mu(X_k/X_{k-1})\geq a$, and thus $Y_l\subset Y^{(a)}$ by definition.\\[1ex]
In case $k=0$ there is nothing to show. For $k=1$, suppose $f(X_1)\subset Y_l\setminus Y_{l-1}$. Then $f$ induces a non-zero map between the semistable representations $X_1$ and $Y_l/Y_{l-1}$, showing $\mu(X_1)\leq\mu(Y_l/Y_{l-1})$ as claimed.\\[1ex]
For general $k$, suppose that $f(X_k)\subset Y_l\setminus Y_{l-1}$, and consider the diagram
$$\begin{array}{ccccccccc}0&\rightarrow&X_{k-1}&\stackrel{\alpha}{\rightarrow}&X_k&\rightarrow&X_k/X_{k-1}&\rightarrow&0\\
&&&&f\downarrow&&&&\\ 
0&\rightarrow&Y_{l-1}&{\rightarrow}&Y_l&\stackrel{\beta}{\rightarrow}&Y_l/Y_{l-1}&\rightarrow&0\end{array}$$
If $\beta f\alpha$ equals $0$, the map $f$ induces a non-zero map $X_k/X_{k-1}\rightarrow Y_l/Y_{l-1}$ between semistable representations, and thus $\mu(X_k/X_{k-1})\leq\mu(Y_l/Y_{l-1})$ as desired. If $\beta f\alpha$ is non-zero, we have $f(X_{k-1})\subset Y_l\setminus Y_{l-1}$, and we can conclude by induction that $\mu(Y_l/Y_{l-1})\geq\mu(X_{k-1}/X_{k-2})>\mu(X_k/X_{k-1})$, which again gives the desired estimate.\end{proof}

We call the slopes $\mu(X_1/X_0),\ldots,\mu(X_s/X_{s-1})$ in the unique HN filtration of $X$ the weights of $X$.

\begin{definition} Given $a\in{\bf Q}$, define ${\cal T}_a$ as the class of all representations $X$ all of whose weights are $\geq a$, and define ${\cal F}_a$ as the class of all representations $X$ all of whose weights are $<a$.
\end{definition}

\begin{lemma} For each $a\in{\bf Q}$, the pair $({\cal T}_a,{\cal F}_a)$ defines a torsion pair in ${\rm mod}\, kQ$. For $a<b$, we have ${\cal T}_a\supset{\cal T}_b$ and ${\cal F}_a\subset{\cal F}_b$.
\end{lemma}

\begin{proof} Assume $X\in{\cal T}_a$ and $Y\in{\cal F}_a$. In the ${\bf Q}$-indexed HN filtration, we thus have $X^{(b)}=X$ for all $a\leq b$, and $Y^{(b)}=0$ for all $a<b$. By functoriality, any morphism $f:X\rightarrow Y$ is already zero, proving ${\rm Hom}({\cal T}_a,{\cal F}_a)=0$.\\[1ex]
Now assume ${\rm Hom}(X,{\cal F}_a)=0$ for some representation $X$. Suppose $X$ has a weight strictly less than $a$, then certainly the slope of the (semistable) top HN factor $X/X_{s-1}$ is strictly less than $a$, too, thus it belongs to ${\cal F}_a$. But the projection map $X\rightarrow X/X_{s-1}$ is non-zero, a contradiction. Thus, $X$ belongs to ${\cal T}_a$. Finally, assume ${\rm Hom}({\cal T}_a,Y)=0$ for some representation $Y$. If $Y$ has a weight $\geq a$, then certainly the slope of its (semistable) scss $Y_1$ is $\geq a$. Thus $Y_1$ belongs to ${\cal T}_a$. But the inclusion $Y_1\rightarrow Y$ is non-zero, a contradiction. Thus, $Y$ belongs to ${\cal F}_a$. The inclusion properties of the various torsion and free classes follows from the definitions.\end{proof}

\section{Further geometric properties of moduli spaces}\label{further}

\subsection{The choice of $\Theta$}

In general, the choice of a ``suitable'' stability $\Theta$ for a given quiver, or a given dimension vector, is a very subtle problem. We will not consider this in general, but only work out the possible choices in some examples.\\[1ex]
First note that there are two operations on functionals $\Theta$ as above which do not change the classes of (semi-)stable representations:
\begin{itemize}
\item first, we can multiply $\Theta$ by a non-negative integer;
\item second, we can add an integer multiple of the functional $\dim$ to $\Theta$.
\end{itemize}

This shows immediately that all choices of $\Theta$ for the multiple-loop quiver $L_m$ are equivalent to the choice $\Theta=0$. In this case, we know that all representations are semistable (of slope $0$), and that the stable representations are precisely the simple ones. Thus in this case, the only relevant moduli spaces are the moduli $M_d^{\rm st}(L_m)$ of $d$-dimensional simples.\\[1ex]
For generalized Kronecker quivers, the above operations on stability conditions allow reduction to three choices, namely $\Theta_1(d,e)=0$, $\Theta_2(d,e)=e$ and $\Theta_3(d,e)=d$. The first two only lead to trivial notions of stability: for $\Theta_1=0$, all representations are semistable, and the stables are precisely the simples $S_1$, $S_2$. For $\Theta_2$, assume that $M$ is semistable of dimension vector $(d,e)$. If $d\not=0\not=e$, consideration of a subrepresentation $S_2\subset M$ yields a contradiction. Thus, the only semistables are direct sums of copies of either $S_1$ or $S_2$, and these simple are the only stables. Thus, the only non-trivial choice of stability is $\Theta_3$. In this case, a representation $M$ is semistable if and only if for all non-zero subspaces $U$ of $M_i$, we have
$$\dim\sum_{k=1}^mM_{\alpha_k}(U)\geq\frac{\dim M_j}{\dim M_i}\dim U,$$
and it is stable if this inequality is strict for all proper non-zero subspaces.

\subsection{Stable representations and Schur representations}\label{stableschur}

From Lemma \ref{basicsemistable}, we know that every stable representation is a Schur representation. We will now show in an example (for a generalized Kronecker quiver), that in general there exists no choice $\Theta$ of stability making all Schur representations of a fixed dimension vector stable. It is also not possible in general to choose a $\Theta$ for a given Schur representation $M$ such that $M$ becomes $\Theta$-stable.\\[1ex]
Continuing the example from section \ref{ex5k}, we have

\begin{lemma} The representation $M(\lambda,\mu)$ for $(\lambda,\mu)\not=0$ is semistable if and only if it is stable if and only if $\lambda\not=0$.
\end{lemma}

\begin{proof} The representation $M(0,1)$ admits a subrepresentation of dimension vector $(1,1)$, contradicting (semi-)stability. If $\lambda\not=0$, the only possible dimension vectors of subrepresentations of $M(\lambda,\mu)$ are easily worked out as $$(0,1),(0,2),(0,3),(1,2),(1,3),(2,3),$$
proving stability.\end{proof}

Therefore, the notion of stability chooses one of the two representations $M(0,1)$, $M(1,0)$ as ``more canonical'', due to its submodule structure.\\[1ex]
As another example, we consider the $4$-subspace quiver $S_4$ with dimension vector $d=i_1+i_2+i_3+i_4+2)$ and stability $\Theta=j^*$. A representation is given by four vectors $v_1,v_2,v_3,v_4$ in $k^2$, which we will assume to be non-zero. Then this representation is stable of no two of the vectors are proportional, and it is semistable if no three of them are proportional. It is indecomposable if it is semistable, and the four vectors cannot be grouped into two pairs of proportional ones. The stable representations therefore admit a normal form $(e_1,e_2,e_1+e_2,e_1+\lambda e_2)$ for $\lambda\not=0,1$. The moduli space $M_d^{\rm sst}(S_4)$ is isomorphic to a projective line ${\bf P}^1$ via the map
$R_d^{\rm sst}(Q)\rightarrow {\bf P}^1$ given by
$$(v_1,v_2,v_3,v_4)\mapsto(\det[v_1|v_2]\det[v_3|v_4]:\det[v_1|v_4]\det[v_3|v_2])$$
(see section \ref{coord} for more details). In this way, we realize $M_d^{\rm st}(K_4)$ is the open subset ${\bf P}^1\setminus\{0,1,\infty\}$. But the moduli space cannot distinguish between all the indecomposables: for example, all of the three isomorphism classes
$$(e_1,e_1,e_2,e_2),\; (e_1,e_1,e_2,e_1+e_2),\, (e_1,e_2,e_1+e_2,e_1+e_2)$$
are sent to the point $(0:1)$ of ${\bf P}^1$. The first of these is decomposable into a direct sum $U_1\oplus U_2$ of three-dimensional representations. The second and the third are indecomposable, the second being a non-trivial extension of $U_1$ by $U_2$, the third being a non-trivial extension of $U_2$ by $U_1$.\\[1ex]
As already suggested above, one may ask whether for a given Schur representation, there always exists a stability making this representation stable. The above example shows that this is not possible in general!
The point of view of the present paper is not to worry about the different choices of stability (and even not about the question whether a particular moduli space is non-empty), but to try to formulate results which hold for arbitrary choices of stability.

\subsection{Existence of stable representations}

An obvious question in relation to the choice of $\Theta$ is the following: under which conditions is $M_d^{\rm sst}(Q)$ (resp. $M_d^{\rm st}(Q)$) non-empty? For the semistable representations, the Harder-Narasimhan filtration yields a recursive criterion, see \cite{RHNS}.

\begin{proposition} $M_d^{\rm sst}(Q)$ is non-empty if and only if there exists no non-trivial decomposition $d=d^1+\ldots+d^s$ fulfilling the following three conditions:
\begin{itemize}
\item $M_{d^k}^{\rm sst}(Q)\not=\emptyset$ for all $k=1,\ldots,s$,
\item $\mu(d^1)>\ldots>\mu(d^s)$,
\item $\langle d^k,d^l\rangle=0$ for all $k<l$.
\end{itemize}
\end{proposition}

Criteria for non-emptyness of both $M_d^{\rm st}(Q)$ and $M_d^{\rm sst}(Q)$ can be formulated using the concept of generic subrepresentations \cite{Scho}; these criteria are also highly recursive. We write $e\hookrightarrow d$ if the set of representations of dimension vector $d$ admitting a subrepresentation of dimension vector $e$ is dense in $R_d(Q)$. Whether this condition holds for two given dimension vectors can be determined recursively:
\begin{theorem} $e\hookrightarrow d$ if and only if $\langle e',d-e\rangle\geq 0$ for all $e'\hookrightarrow e$.
\end{theorem}

Based on this notion, we have the following criterion from \cite{Scho}:

\begin{theorem} We have
\begin{enumerate}
\item $M_d^{\rm sst}(Q)\not=\emptyset$ if and only if $\mu(e)\leq \mu(d)$ for all $e\hookrightarrow d$,
\item $M_d^{\rm st}(Q)\not=\emptyset$ if and only if $\mu(e)<\mu(d)$ for all $e\not=d$ such that $0\not=e\hookrightarrow d$.
\end{enumerate}
\end{theorem}

A ``less recursive'' criterion is formulated in \cite{AL}.

\begin{theorem} Assume $d=\sum_{k=1}^s m_kd^k$ can be written as a positive combination of dimension vectors $d^k$ such that $M_{d^k}^{\rm st}(Q)\not=\emptyset$ for all $k$. Then $M_d^{\rm st}(Q)\not=\emptyset$ if and only if $(m_1,\ldots,m_s)$ is the dimension vector of a simple representation of the quiver with vertices $i_1,\ldots,i_s$ and $\delta_{k,l}-\langle d^k,d^l\rangle$ arrows between each pair of vertices $i_k,i_l$.
\end{theorem}

This reduces the problem to the question of existence of simple representations, which is solved in \cite{LBP}:

\begin{theorem} We have $M_d^{\rm simp}(Q)\not=\emptyset$ if and only if ${\rm supp}(d)$ is a quiver of type $\widetilde{A}_n$ with cyclic orientation and $d_i=1$ for all $i\in{\rm supp}(d)$, or ${\rm supp}(d)$ is not of the above type and $\langle d,i\rangle\leq 0\geq\langle i,d\rangle$ for all $i\in{\rm supp}(d)$.
\end{theorem}

\subsection{Universal bundles}

It is a general philosophy in moduli theory that moduli spaces should, most desirably, be equipped with so-called universal (or tautological) bundles.\\[1ex]
As an elementary example, we consider the tautological bundle on the Grassmannian ${\rm Gr}_k(V)$ parametrizing $k$-dimensional subspaces of a vector space $V$, i.e.~we can label the points of ${\rm Gr}_k(V)$ by $k$-dimensional subspaces $U\subset V$.  There exists a vector bundle $\pi:\mathcal{T}\mapsto{\rm Gr}_k(V)$, which is a subbundle of the trivial bundle $p_1:{\rm Gr}_k(V)\times V\mapsto {\rm Gr}_k(V)$ with the following property: $\pi^{-1}(U)\subset\{U\}\times V$ consists of all $(\{U\},v)$ such that $v\in U$.\\[1ex]
Similarly, in the context of quiver moduli, we ask for the following: let $M_d^{\rm st}(Q)$ be a moduli space of stable representations of $Q$ of dimension vector $d$. We want to define vector bundles $\pi_i:\mathcal{V}_i\mapsto M_d^{\rm st}(Q)$ of rank $d_i$ for all $i\in I$ and vector bundle maps $\mathcal{V}_\alpha:\mathcal{V}_i\rightarrow\mathcal{V}_j$ for all arrows $\alpha:i\rightarrow j$ in $Q$ such that the following holds:\\[1ex]
consider the fibres $\pi_i^{-1}(M)$ for some point $M\in M_d^{\rm st}(Q)$. Then the representation of $Q$ induced on the vector spaces $\pi_i^{-1}(M)$ by the maps $\mathcal{V}_\alpha$ is isomorphic to $M$.\\[1ex]
The idea behind the construction of this universal representation (see \cite{King}) is quite obvious: we consider the trivial vector bundles $R_d^{\rm st}(Q)\times M_i\rightarrow M_i$ and the vector bundle maps $(M,m_i)\mapsto (M,M_\alpha(m_i))$ for $\alpha:i\rightarrow j$. These become $G_d$-bundles via the standard action of $G_d$ on each $M_i$. We want these bundles to descend to bundles on the quotients $M_d^{\rm st}(Q)=R_d^{\rm st}(Q)/G_d$. This works only if the induced action of the stabilizer of a point on the fibres of the bundle is trivial, which is not true: consider the action of scalar matrices in $G_d$, which are the stabilizers of stable representations. The way out of this problem is to twist the $G_d$-action on $R_d^{\rm st}(Q)\times M_i$ by a character $\chi$, which necessarily has to take the value $\lambda^{-1}$ on a scalar matrix $\lambda\in G_d$. Such a character exists if and only if the dimension vector $d$ is coprime in the sense that ${\rm gcd}(d_i\, :\, i\in I)=1$: we can then choose a tuple $s_i$ of integers such that $\sum_{i\in I}s_id_i=1$ and define $\chi((g_i)_i)=\prod_{i\in I}\det(g_i)^{-s_i}$.\\[1ex]
It is likely that no universal bundle on $M_d(Q)$ exists in case $d$ is not coprime, but there is no proof of this yet.

\subsection{Coordinates}\label{coord}

We now turn to the question of coordinates for quiver moduli. By definition, we have
$$M_d^{\rm sst}(Q)={\rm Proj}(\bigoplus_{n\geq 0}k[R_d]^{G_d,\chi_\Theta^n}),$$
thus knowledge of generating (semi-)invariants provides coordinates for the moduli in the following sense:\\[1ex]
Let $R=\bigoplus_{n\geq 0}R_n$ be such a semi-invariant ring. Then $R_0$ is finitely generated (being the invariant ring for the action of $G_d$ on $R_d$) by, say, $f_1,\ldots,f_s$. Consider $R$ as an $R_0$-algebra. This is again finitely generated, since the ring of semi-invariants (more precisely, the underlying non-graded ring) can be viewed as the ring of invariants for the smaller group ${\rm Ker}(\chi_\Theta)$. Since the ${\rm Proj}$-construction is not sensitive to ``thinning'' $R$, i.e.~replacing $R$ by $R^{(k)}=\bigoplus_{n\geq 0}R_{kn}$ for $k\geq 1$, we can choose generators $g_0,\ldots,g_t$, homogeneous of some degree $k\geq 1$, for $R^{(k)}$. Then ${\rm Proj}(R)$ admits an embedding into ${\bf A}^s\times{\bf P}^t$ dual to the surjection $k[x_1,\ldots,x_s][y_0,\ldots,y_t]\rightarrow R$ attaching $f_i$ to $x_i$ and $g_j$ to $y_j$.\\[1ex]
This procedure can be carried out in principle for $R$ the ring of semi-invariants of the action of $G_d$ on $R_d$ with respect to a character $\chi_\Theta$. We start with the ring of invariants:

\begin{theorem}[Le Bruyn-Procesi] The ring of invariants for the action of $G_d$ on $R_d$ is generated by traces along oriented cycles, i.e.~for a cycle $\omega=\alpha_u\ldots\alpha_1$ given by
$$i_1\stackrel{\alpha_1}{\rightarrow}i_2\stackrel{\alpha_2}{\rightarrow}\ldots\stackrel{\alpha_{u-1}}{\rightarrow}i_u\stackrel{\alpha_u}{\rightarrow}i_1$$
in $Q$, we consider the function ${\rm tr}_\omega$ assigning to a representation $M=(M_\alpha)_\alpha$ the trace ${\rm tr}(M_{\alpha_u}\cdot\ldots\cdot M_{\alpha_1})$.
\end{theorem}

Taking enough traces along oriented cycles, we thus get an embedding of $M_d^{\rm (s)simp}(Q)$ into an affine space.\\[1ex]
To formulate a similar statement for semi-invariants, we have to introduce some additional notation. We start with the case of quivers without oriented cycles. Given representations $M$ and $N$ of $Q$, we can compute ${\rm Hom}(M,N)$ and ${\rm Ext}^1(M,N)$ as the kernel and cokernel, respectively, of the map
$$d_{M,N}:\bigoplus_{i\in I}{\rm Hom}_k(M_i,N_i){\rightarrow}\bigoplus_{\alpha:i\rightarrow j}{\rm Hom}_k(M_i,N_j)$$
given by
$$d_{M,N}((f_i)_i)=(N_\alpha f_i-f_j M_\alpha)_{(\alpha:i\rightarrow j)}.$$
In case $$\langle\dimv M,\dimv N\rangle=\dim{\rm Hom}(M,N)-\dim{\rm Ext}^1(M,N)=0,$$
we thus have a map $d_{M,N}$ between vector spaces of the same dimension, and we can consider its determinant $c(M,N)=\det d_{M,N}$.
Varying $M$ and $N$ in their respective spaces of representations, we get a polynomial function $$c:R_d(Q)\times R_d(Q)\rightarrow k,$$
which is in fact a semi-invariant function for the action of $G_d\times G_e$. We can also fix the representation $M$ and vary the representation $N$ to obtain a semi-invariant function $c_M$ on $R_e(Q)$. The following is proved in \cite{SV}:

\begin{theorem}[] The functions $c_M$ for representations $M$ such that $\langle\dimv M,d\rangle=0$ generate the ring of semi-invariants on $R_d$.
\end{theorem}

In case $Q$ has oriented cycles, we need a more general version unifying the above two theorems. Choose an arbitrary map $v:P\rightarrow Q$ between finitely generated projective representations $$P=\bigoplus_{i\in I}P_i^{a_i}\mbox{ and }Q=\bigoplus_{i\in I}P_i^{b_i}$$ ($P_i$ denoting the projective indecomposables associated to the vertex $i\in I$), such that $$\sum_{i\in I}(a_i-b_i)d_i=0.$$
Then the induced map $${\rm Hom}(v,M):{\rm Hom}(Q,M)\rightarrow{\rm Hom}(P,M)$$ is a map between vector spaces of the same dimension, and again we can consider its determinant $c_v(M)=\det{\rm Hom}(v,M)$.
This defines a semi-invariant function on $R_d(Q)$. Again by \cite{SV}, we have:

\begin{theorem} The ring of semi-invariant functions on $R_d(Q)$ is generated by the functions $c_v$.
\end{theorem}

Note again that the full semi-invariant ring, which is described by the above theorems, is not of the type we considered in the definition of moduli spaces, i.e.~not associated to a single stability function $\Theta$. The geometric object that is described by the ${\rm Proj}$ of this ring is the quotient of the stable points in $R_d$ by the action of the smaller group $\prod_{i\in I}{\rm SL}(M_i)$. It parametrizes quiver representations $M$ together with fixed volume forms of all vector spaces $M_i$, under isomorphisms preserving the volume forms.

\subsection{An example -- subspace quivers}\label{exsub}

We consider an example to illustrate the above strategy for computing coordinates for quiver moduli and to show the difference between the full ring of semi-invariants and the ring of semi-invariants associated to a fixed character. Consider the $m$-subspace quiver $S_m$ and dimension vector $(1,\ldots,1,2)$. A representation is given as an $m$-tuple $(v_1,\ldots,v_m)$ of vectors in $k^2$, on which the group ${\rm GL}_2(k)\times(k^*)^m$ acts via $$(g,x_1,\ldots,x_m)*(v_k)_k=(\frac{1}{x_k}gv_k)_k.$$
It is easy to see that the full ring of semi-invariants $R$ is generated by the functions
$$D_{kl}=\det[v_k|v_l]\mbox{ for }1\leq k<l\leq m.$$
These functions fulfill the Pl\"ucker relations $$D_{ik}D_{jl}=D_{ij}D_{kl}+D_{ik}D_{jl}$$
for all $i<j<k<l$. Thus, the ${\rm Proj}$ of this ring is nothing but the Grassmannian of $2$-planes ${\rm Gr}_2(k^m)$.\\[1ex]
The function $D_{kl}$ is a semi-invariant for the character $$\chi(g,x_1,\ldots,x_m)=\frac{\det(g)}{x_kx_l}.$$
Let us consider the ``most symmetric'' stability $\Theta=(0,\ldots,0,-1)$. Now a monomial $\prod_{i<j}D_{ij}^{m_{ij}}$ belongs to $k[R_d]^{G_d}_{\chi_\Theta}$ if and only if for any $i=1,\ldots,n$, we have
$$\sum_{j<i}m_{ji}+\sum_{j>i}m_{ij}=2k$$ for some $k\geq 1$. It turns out a minimal system of generators becomes rather large even for small values of $n$.\\[1ex]
We consider the particular case $n=5$. One can see directly that the semi-invariant ring is generated by the following functions:
$$c=D_{12}D_{23}D_{34}D_{45}D_{51},\;\;\; x_i=D_{i,i+1}D_{i,i+4}D_{i+4,i}D_{i+2,i+3}^2\mbox{ for }i\in{\bf Z}_5.$$
The Pl\"ucker relations give the following (defining) relations between the generators:
$$x_ix_{i+1}=c^2+cx_{i+3}\mbox{ for }i\in{\bf Z}_5.$$
Consequently, the moduli space $M_d^{\rm (s)st}(S_5)$ is the surface in ${\bf P}^5$ with coordinates $(x_1:x_2:x_3:x_4:x_5:c)$ determined by the five equations above.\\[1ex]
The case $m=7$ was worked out in \cite{Albrandt}, using results of \cite{HMS}: the above methods yields an embedding of the $4$-dimensional moduli space into ${\bf P}^{35}$, determined by $58$ defining equations.\\[1ex]
From these examples we can see that, even in simple examples, coordinatization of the moduli spaces leads to difficult explicit calculations in commutative algebra.\\[1ex]
Another question is whether, even if we have explicit generators and defining relations, this is helpful for studying the moduli spaces, since it is difficult to extract global geometric information from defining equations.

\section{Cohomology and cell decompositions}\label{sectbettigross}

One of the possible directions towards a study of the global geometry of quiver moduli pursued by the author in \cite{ERSM, MR, RHNS, RFQM, RCNH, RCRP, RLoc, RW} is the determination of Betti numbers of quiver moduli. We will first consider the question why knowledge of the Betti numbers should be interesting for such a study.\\[1ex]
Whenever we use cohomology of varieties, we will work over the base field $k={\bf C}$ of the complex numbers, and we will consider all quasiprojective varieties with their ${\bf C}$-topology, induced from the natural ${\bf C}$-topology on complex projective spaces. Then we consider singular cohomology (or singular cohomology with compact support) with coefficients in ${\bf Q}$, disregarding all potential torsion phenomena. We then denote by $b_i(X)=\dim_{\bf Q}H^i(X,{\bf Q})$ the $i$-th Betti number of $X$ (for arbitrary base fields $k$ as before, $\ell$-adic cohomology should be considered; standard comparison theorems guarantee the compatibility of these two approaches).
 
\subsection{Cell decompositions}

\begin{definition} A variety $X$ is said to admit a cell decomposition if there exists a filtration
$$X=X_0\supset X_1\supset \ldots\supset X_t=\emptyset$$
by closed subvarieties, such that the successive complements $X_{j-1}\setminus X_j$ for $j=1,\ldots,t$ are isomorphic to affine spaces ${\bf A}^{d_j}$.
\end{definition}

This notion is not to be confused with the topological notion of cell decomposition, for example in the context of CW-complexes. In the literature (for example \cite{CG}) one also finds a variant of this notion, where the successive complements are only required to be isomorphic to disjoint unions of affine spaces; further refinement of such a filtration leads to one in the above sense.\\[1ex]
Examples of such varieties are provided by affine space itself, by projective spaces (where $X_i$ consists of all points $(x_0:\ldots:x_n)$ in projective $n$-space such that the first $i$ coordinates are zero, for $i=0,\ldots,n-1$), Grassmannians, etc..\\[1ex]
If $X$ admits a cell decomposition, then all odd cohomology of $X$ vanishes, and the $2i$-th Betti number is given as the number of $i$-dimensional cells, i.e.~the number of indices $j$ such that $X_{j-1}\setminus X_j\simeq {\bf A}^i$.

\subsection{The importance of cell decompositions for quiver moduli}

Now assume that $M_d^{\rm st}(Q)$ is a quiver moduli admitting a universal representation ${\cal V}$, and assume that $M_d^{\rm st}(Q)$ admits a cell decomposition. Then we can write $$M_d^{\rm st}(Q)=\bigcup_{k=1}^nY_k\mbox{, where }Y_k\simeq{\bf A}^{d_k}.$$
The restriction of each vector bundle ${\cal V}_i$ to each $Y_k$ is a vector bundle over an affine space, and thus trivial. This means that we can find isomorphisms $$\phi_{ik}:Y_k\times M_i\rightarrow {\cal V}_i|_{Y_k}.$$
For any arrow $\alpha:i\rightarrow j$ we can then consider the composite map $$f_\alpha=\phi_{jk}^{-1}\circ{\cal V}_\alpha|_{Y_k}\circ\phi_{ik}:Y_k\times M_i\rightarrow Y_k\times M_j.$$
For any point $x$ in ${\bf A}^{d_k}\simeq Y_k$, the restriction $f_\alpha(x)$ of $f_\alpha$ to the fibres over $x$ is a linear map from $M_i$ to $M_j$ by definition. Thus, the tuple $(f_\alpha(x))_\alpha$ defines a quiver representation $M(x)$. The collection of all $M(x)$ for $x\in Y_k$ thus give a normal form for all the quiver representations belonging to $Y_k$.\\[1ex]
We see that existence of a cell decomposition of the moduli space $M_d^{\rm st}(Q)$ gives an explicit parametrization of all isomorphism classes of stable representations of dimension vector $d$, together with explicit normal forms. This is of course very desirable in view of Problem \ref{basicproblem}.\\[1ex]
One may conjecture that such cell decompositions exist for the moduli of stable representations whenever $d$ is coprime for $\Theta$. Instances of this conjecture will be proved later for certain (very special) dimension vectors of generalized Kronecker quivers (see section \ref{bb}), and for the framed versions of moduli spaces in section \ref{hilbert}.

\subsection{Negative examples and discussion}\label{ned}

An interesting testing case for the above conjecture is provided by the surface considered in section \ref{exsub}: it is not clear whether this projective rational surface admits a cell decomposition.\\[1ex]
Turning to the non-coprime case, one cannot hope for cell decompositions to exist in general. Consider again the example of the $4$-subspace quiver $S_4$ and dimension vector $(1,1,1,1,2)$. We have seen in section \ref{stableschur} that the moduli of stables is isomorpic to ${\bf P}^1$ minus three points in this case. This space definitely has (two-dimensional) first cohomology, so it cannot admit a cell decomposition (since this implies vanishing of the odd Betti numbers). Another example is provided by two-dimensional simple representations of the two-loop quiver $L_2$.\\[1ex]
One might suspect that the problem is caused by the missing semistables (or semisimples). This is, however, not the case. Considering the dimension vector $(2,2)$ for the five-arrow Kronecker quiver, we get the counting polynomial $$t^{13}+t^{12}+3t^{11}+2t^{10}+3t^9+t^8+t^7-t^6+t^3+t^2+t+1$$ (see section \ref{arithm}) for $M_d^{\rm sst}(Q)$, thus there cannot exist a cell decomposition, as will be explained in section \ref{cdrp} (a polynomial counting points over finite fields of a variety admitting a cell decompositions necessarily has nonnegative coefficients).\\[1ex]
Despite these negative results, there are several possible variants of the problem: the first possibility is to ask for a torus decomposition, i.e.~we relax the defining condition of a cell decomposition and ask the successive complements to be isomorphic to tori instead. We will see in section \ref{simples} that this is supported by the conjecture \ref{conjpos}.\\[1ex]
Another option is to look at variants of quiver moduli, and to ask for these spaces to admit a cell decomposition. For moduli of simple representations, we will consider the noncommutative Hilbert schemes in section \ref{hilbert}, and cell decompositions will be constructed. In general, the smooth models of section \ref{sm} provide candidates -- again one might conjecture that they always admit cell decompositions.\\[1ex]
The most desirable variant of the original moduli spaces we would like to have is a desingularization of the moduli of semistables, or, in other words, a ``compactification'' of the moduli of stables. By this we mean a smooth variety $X$ admitting a projective birational morphism to $M_d^{\rm sst}(Q)$ (then $X$ is projective over $M_d^{\rm ssimp}(Q)$, and contains $M_d^{\rm st}(Q)$ as an open subset). Existence of cell decompositions for such desingularisation poses a difficult problem: suppose that $X$ has a cell decomposition. In particular, we have $X\setminus X_1\simeq{\bf A}^{d_1}$, i.e.~there exists an open subset which is isomorphic to affine space -- in other words, the variety $X$ is rational. Since $X$ maps to $M_d^{\rm sst}(Q)$ birationally, this would imply that the latter is rational, too. But it is shown in \cite{SchoB} that all quiver moduli are birational to moduli of simple representations of the multiple-loop quiver, and in this case rationality is a long-standing open problem, see \cite{LB0}.\\[1ex]
One other possible relaxation of the notion of cell decomposition is therefore to ask for an orbifold decomposition, i.e.~the successive complements should look like quotients of affine space (or a torus) by a finite group action.\\[1ex]
There is one notable exception to this problem of construction of smooth compactifications. Namely, the moduli space $M_2^{\rm ssimp}(L_m)$ has been desingularized in \cite{SN}. This desingularization is analysed in detail in \cite{Ol}. In particular, it is shown there that the desingularization admits a cell decomposition. Moreover, the obstruction to a generalization of the construction of \cite{SN} to higher dimensions is studied in detail, refining \cite{LBR}.\\[1ex]
One of the problems in constructing cell decompositions is that there are only few general techniques for doing so. One is the Bialynicki-Birula method, to be discussed in section \ref{bb}.

\subsection{Betti numbers and prediction of cell decompositions}

One of the main motivations for computing and studying Betti numbers of quiver moduli can be seen in the following line of reasoning, based on the above discussion: suppose the Betti numbers of the moduli space in question are computed, and that they admit some combinatorial interpretation (for example numbers of certain types of trees in the case of Hilbert schemes in section \ref{hilbert}). Then this gives a prediction for a combinatorial parametrization of the cells in a cell decomposition, and sometimes actually a construction of the cells.

\subsection{Betti numbers for quiver moduli in the coprime case}\label{sectbetti}

We review the main results of \cite{RHNS}.

\begin{definition}\label{defpdq} Given $Q$, $d$ and $\Theta$ as before, we define the following rational function in $q$:

$$P_d(q)=\sum_{d^*}(-1)^{s-1}q^{-\sum_{k\leq l}\langle d^l,d^k\rangle}\prod_{k=1}^s\prod_{i\in I}\prod_{j=1}^{d_i}(1-q^{-j})^{-1}\in{\bf Q}(q),$$
where the sum ranges over all tuples $d^*=(d^1,\ldots,d^s)$ of dimension vectors such that
\begin{itemize}
\item $d=d^1+\ldots+d^s$
\item $d^k\not=0$ for all $k=1,\ldots,s$,
\item $\mu(d^1+\ldots+d^k)>\mu(d)$ for all $k<s$.
\end{itemize}
\end{definition}

We will see in section \ref{apphall} how the definition of this function is motivated (basically, we have $$P_d(q)=\frac{|R_d^{\rm sst}(Q)({\bf F}_q)|}{|G_d({\bf F}_q)|}$$ for any finite field ${\bf F}_q$).

\begin{theorem}\label{betti} If $d$ is coprime for $\Theta$, then $(q-1)\cdot P_d(q)=\sum_i\dim_{\bf Q}H^{i}(M_d,{\bf Q})q^{i/2}$.
\end{theorem}

In particular, this theorem reproves a result of \cite{KW} that there is no odd cohomology of $M_d(Q)$ in the coprime case: namely, the left hand side of the formula in the theorem is a rational function in $q$, so the right hand side is so, too, and all potential contributations to half-powers of $q$ -- coming from the odd cohomology -- have to vanish.\\[1ex]
A drawback of this formula is that, although we know a priori that the result is a polynomial in $q$, all summands are only rational functions in $q$, with denominators being products of terms of the form $(1-q^i)$. In particular, we cannot specialize the formula to $q=1$, which would be very interesting because then the Poincare polynomial specializes to the Euler characteristic.\\[1ex]
We also do not get a ``positive'' formula in the sense that one can see directly from the summands that the coefficients of the resulting polynomial have to be nonnegative integers. In contrast, the formula for $P_d(q)$ involves signs. We will see positive formulas for the Betti numbers of quiver moduli in special cases, namely in section \ref{lockronecker} for generalized Kronecker quivers, and in section \ref{hilbert} for Hilbert schemes.\\[1ex]
Nevertheless, the above formula gives rise to a fast algorithm for computing the Betti numbers (which was further optimized in \cite{Weist} to compute the Euler characteristic of moduli of generalized Kronecker quivers). This serves as an important tool for computer experiments which motivated many of the developments described below.

\subsection{Asymptotic aspects}

One can argue that, in studying quiver moduli qualitatively (i.e.~studying common features enjoyed by all quiver moduli, in contrast to determination of special features of particular ones) one should not consider a fixed dimension vector $d$, but consider either all of them at the same time (see results on generating series over all $d$ such that $\mu(d)=d$ in Theorems \ref{formulanumbstab} and \ref{cohomsm}), or consider the behaviour of the moduli for large $d$. The latter case is what is considered in this section.\\[1ex]
We first consider a very simple instance of this principle: in case $d$ is coprime for $\Theta$, we will see in section \ref{apphall} that the Poincare polynomial of singular cohomology is given by $$(q-1)\cdot\frac{|R_d^{\rm sst}(Q)({\bf F}_q)|}{|G_d|({\bf F}_q)|}.$$
Since $R_d^{\rm sst}(Q)$ is open in the affine space $R_d^{\rm sst}$, we know that the counting polynomial is of degree $\dim R_d(Q)$. For large enough coprime $d$, the number of non-semistable representations should be ``small'' compared to the number of all representations, so that $$(q-1)\frac{|R_d(Q)|}{|G_d|}$$ should be a good approximation to the Poincare polynomial. This number equals (by a direct calculation)
$$q^{1-\langle d,d\rangle}(1-q^{-1})\prod_{i\in I}\prod_{k=1}^{d_i}(1-q^{-k})^{-1}.$$
A very surprising route towards predictions of the asymptotic behaviour of quiver moduli opens up in connection to methods from string theory, see e.g. \cite{BGL,Denef,DenefMoore}. The idea is, very roughly, to view e.g.~moduli of representations of generalized Kronecker quivers as moduli spaces of the possible states of strings between branes. These should form a microscopic model for the behaviour of macroscopic physical systems like e.g.~certain types of black holes. Known or expected properties of such physical systems then yield predictions on the microscopic system (i.e.~the quiver moduli) ``in the large'', i.e.~for large values of the dimension vector.\\[1ex]
M. Douglas made the following conjecture for generalized Kronecker quivers $K_m$: for large dimension vectors $(d,e)$ ($d$ and $e$ coprime), the logarithm of the Euler characteristic $\log \chi(M_{(d,e)}(Q))$ should depend continuously on the ratio $e/d$. A more precise formulation can be given as follows:\\[1ex]
there should exist a continuous function $f:{\bf R}_{\geq 0}\rightarrow{\bf R}$ with the following property: for any $r\in{\bf R}_{\geq 0}$ and any $\varepsilon>0$, there exist $\delta>0$ and $N\in{\bf N}$ such that for coprime $(d,e)$ with $d+e>N$ and $|e/d-r|<\delta$, we have $$|f(r)-\frac{\log\chi(M_{(d,e)}(K_m))}{d}|<\varepsilon.$$

This conjecture is extremely surprising mathematically, since there are no general geometric or representation-theoretic techniques to relate moduli spaces $M_d(Q)$ and $M_e(Q)$ for ``close'' coprime dimension vectors $d$ and $e$. Nevertheless, computer experiments \cite{Weist} using the above mentioned algorithm for computation of Betti numbers give substantial evidence for this conjecture. A posteriori, it turns out that, if such a function $f$ exists, it already is uniquely determined up to a constant as $$f(r)=C\cdot\sqrt{r(m-r)-1}.$$
This can be seen using natural identifications of moduli spaces (which are special to moduli for generalized Kronecker quivers). Namely, the natural duality, resp. the reflection functors, yield isomorphisms
$$M_{(d,e)}^{\rm st}(K_m)\simeq M_{(e,d)}^{\rm st}(K_m)\mbox{, resp. }M_{d,e}^{\rm st}(K_m)\simeq M_{(md-e,d)}^{\rm st}(K_m).$$
These identifications translate into functional equations for the function $f$ (if it exists), which already determine it up to a scalar factor.\\[1ex]
Using localization techniques (see section \ref{loc}), it is possible to obtain exponential lower bounds for the Euler characteristic, thus proving part of the above conjecture; see \cite{RW,Weist2}.\\[1ex]
A slight reformulation of the above (conjectural) formula for the asymptotic behaviour yields a conjecture for arbitrary quivers:\\[1ex]
For every quiver $Q$, there exists a constant $C_Q$ such that for large coprime $d$, we have $\log\chi(M_d(Q))\sim C_Q\sqrt{-\langle d,d\rangle}$.\\[1ex]
It is a very interesting problem to make this more precise. If this conjecture is true in some form, it has the surprising consequence that the Euler characteristic of a quiver moduli ``in the large'' is already determined by it dimension $1-\langle d,d\rangle$!\\[1ex]
In one instance, the exponential behaviour of the Euler characteristic can indeed be proved: we consider the moduli space ${\rm Hilb}_{d,1}(L_m)$ (see section \ref{hilbert} for the definition), or, in other words, the moduli space $M_{(1,d)}^{\rm st}(Q)$ for the quiver $Q$ given by vertices $I=\{i,j\}$ and arrows $$Q_1=\{(\alpha:i\rightarrow j),(\beta_1,\ldots,\beta_m:j\rightarrow j)\}.$$
In this case, the parametrization of a cell decomposition by $m$-ary trees (as a special case of the combinatorial notions of section \ref{hilbert}) yields the following \cite{RCNH}:
$$\chi({\rm Hilb}_{d,1}(L_m))\sim C\cdot d^{-3/2}\cdot(m^m/(m-1)^{(m-1)})^d.$$
In this case, it is even possible to describe the asymptotic behaviour of the individual Betti numbers \cite{RCNH}: for each $d\in{\bf N}$, define a discrete random variable $X_d$ by
$${\bf P}(X_d=k)=\frac{1}{\chi({\rm Hilb}_{d,1}(L_m))}\cdot\dim H^{(m-1)d(d-1)-2k}({\rm Hilb}_{d,1}(L_m),{\bf Q}).$$
Then the sequence of random variables
$$\sqrt{8/(m(m-1))}\cdot d^{-3/2}\cdot X_d$$
admits a continuous limit distributation, the so-called Airy distribution \cite{FL}.

\section{Localization}\label{loc}

The localization principle in topology states that a lot of topological information on a space $X$ can be retrieved from the set of fixed points $X^T$ under the action of a torus $T$ on $X$. For example, $\chi(X)=\chi(X^T)$ for any action of a torus on a quasi-projective variety. See \cite{Ca,CG,EM,GKM}.

\subsection{Localization for quiver moduli}\label{locqm}

We consider the torus $T_Q=({\bf C}^*)^{Q_1}$, i.e.~one copy of ${\bf C}^*$ for each arrow $\alpha$ in $Q$ (in some situations, it is also interesting to consider arbitrary subtori, or for example ${\bf C}^*$ embedded diagonally into $T_Q$, see section \ref{bb}).\\[1ex]
The torus $T_Q$ acts on the path algebra ${\bf C}Q$ via rescaling of the generators $\alpha\in{\bf C}Q$ corresponding to the arrows $\alpha$. By functoriality, $T_Q$ acts on the category of representations and also on all moduli of representations. More precisely, $T_Q$ acts on $R_d(Q)$ (written as a right action) via
$$(M_\alpha)_\alpha(t_\alpha)_\alpha=(t_\alpha M_\alpha)_\alpha.$$
This action naturally commutes with the (left) $G_d$-action on $R_d$. The action fixes (semi-)stable representations, since it does not change the possible dimension vectors of subrepresentations. Thus, the $T_Q$-action on $R_d^{\rm sst}(Q)$ descends to an action on $M_d^{\rm (s)st}(Q)$.\\[1ex]
We will now derive a description of the fixed point set $M_d^{\rm st}(Q)^{T_Q}$ in the case where $d$ is coprime for $\Theta$. Let $M=(M_\alpha)_\alpha$ be a $T_Q$-fixed point in $M_d(Q)$. Thus $M$ is a stable representation, and in particular it has trivial endomorphism ring. Consider the group $$G=\{(g,t)\in PG_d\times T_Q\, :\, gM=Mt\}.$$
By definition, the second projection $p_2:G\rightarrow T_Q$ is surjective. On the other hand, it is injective since the stabilizer of $M$ in $PG_d$ is trivial. Thus, we can invert the second projection, providing us with a map $\varphi:T_Q\rightarrow PG_d$, such that $$\varphi(t)M=Mt\mbox{ for all }t\in T_Q.$$
We can lift $\varphi$ to $G_d$, again denoted by $\varphi$. Denote by $\varphi_i:T_Q\rightarrow {\rm GL}(M_i)$ the $i$-component of $i\in I$. The defining condition of $\varphi$ thus tells us that
$$\varphi_j(t)M_\alpha\varphi_i(t)^{-1}=t_\alpha M_\alpha$$
for all $\alpha:i\rightarrow j$ in $Q_1$ and all $t=(t_\alpha)_\alpha\in T_Q$.
The map $\varphi_i$ can be viewed as a representation of $T_Q$ on $M_i$, which we can decompose into weight spaces, denoting by $X(T_Q)$ the character group of $T_Q$:
$$M_i=\bigoplus_{\lambda\in X(T_Q)}M_{i,\lambda}\mbox{, where }M_{i,\lambda}=\{m\in M_i\, :\, \varphi_i(t)m=\lambda(t)m\mbox{ for all }t\in T_Q\}.$$
The character group $X(T_Q)$ has a basis $e_\alpha$ with $e_\alpha(t)=t_\alpha$, for $\alpha\in Q_1$.
The above equation now yields $$M_\alpha(M_{i,\lambda})\subset M_{j,\lambda+e_\alpha}\mbox{ for all }\alpha:i\rightarrow j\mbox{ and all }\lambda\in X(T_Q).$$
This means that $M$ is automatically a kind of graded representation, which we can view as a representation of a covering quiver, defined as follows:\\[1ex]
let $\widehat{Q}$ be the quiver with set of vertices
$$\widehat{Q}_0=Q_0\times X(T_Q)$$ and arrows
$$\widehat{Q}_1=\{(\alpha,\lambda):(i,\lambda)\rightarrow(j,\lambda+e_\alpha),\, (\alpha:i\rightarrow j)\in Q_1,\,\lambda\in X(T_Q)\}.$$
This covering quiver carries a natural action of $X(T_Q)$ via translation. Then $M$ can be viewed as a representation of $\widehat{Q}$ of some dimension vector $\widehat{d}$ lifting $d$, i.e.~such that $$\pi(\widehat{d}):=\sum_{i,\lambda}\widehat{d}_{(i,\lambda)}=d_i\mbox{ for all }i\in I.$$
This representation is again stable, for the stabiliy $\widehat{\Theta}$ on $\widehat{Q}$ defined by $$\widehat{\Theta}(\widehat{d})=\Theta(\pi(\widehat{d})):$$
by rigidity of the HN-filtration under automorphisms, the HN-filtration is stable under all translation symmetries, thus the filtration descends to a HN filtration of the original representation $M$, which is necessarily trivial by stability of $M$. Conversely, stable representations of $\widehat{Q}$ project to stable representations of $Q$.
From this, we finally get:
\begin{proposition} If $d$ is $\Theta$-coprime, the set of fixed points of $T_Q$ on $M_d^{\rm st}(Q)$ admits a description
$$M_d^{\rm st}(Q)^{T_Q}\simeq\bigoplus_{\widehat{d}}M_{\widehat{d}}^{\rm st}(\widehat{Q}),$$
the union over all translation classes of dimension vectors $\widehat{d}$ for $\widehat{Q}$ which project to $d$.
\end{proposition}
Note that this result is trivial if the quiver $Q$ is a tree, but it yields something non-trivial in case of generalized Kronecker quivers, in case $Q$ has oriented cycles, etc.. For example, in the case of generalized Kronecker quivers, moduli of bipartite quivers appear, as in the following section.

\subsection{Localization for generalized Kronecker quivers}\label{lockronecker}

Consider the three-arrow Kronecker quiver $K_3$. Up to translation, we can assume the support of $\widehat{d}$ to be contained in the connected component $C$ of $\widehat{K_3}$ containing the vertex $(i,0)$. This component has the form of a hexagonal lattice:
$$\ldots\tiny\begin{array}{ccccccccccccc}
&&&&\bullet&&&&\bullet&&&&\\
&&&&\uparrow&&&&\uparrow&&&&\\
&&&&\bullet&&&&\bullet&&&&\\
&&&\swarrow&&\searrow&&\swarrow&&\searrow&&&\\
&&\bullet&&&&\bullet&&&&\bullet&&\\
&&\uparrow&&&&\uparrow&&&&\uparrow&&\\
&&\bullet&&&&(i,0)&&&&\bullet&&\\
&\swarrow&&\searrow&&\swarrow&&\searrow&&\swarrow&&\searrow&\\
\bullet&&&&\bullet&&&&\bullet&&&&\bullet\\
&&&&\uparrow&&&&\uparrow&&&&\\
&&&&\bullet&&&&\bullet&&&&\\
&&&\swarrow&&\searrow&&\swarrow&&\searrow&&&\\
&&\bullet&&&&\bullet&&&&\bullet&&
\end{array}\ldots$$

As a particular example, we consider the dimension vector $d=(2,3)$ for the $m$-arrow Kronecker quiver. In this case, it is easy to work out the stable representations of the covering quiver whose dimension vectors project to $d$. Namely, we find indecomposable representations supported on a subquiver of type $A_5$ with alternating orientation, and indecomposable representations supported on a subquiver of type $D_4$, in both cases corresponding to the maximal roots of the respective Dynkin types. Additionally, we have to chose a labelling of the arrows, considered up to natural symmetry. Thus, we arrive at the following dimension vectors for the covering quiver:
$$\begin{array}{rcrcr}
&&1&\stackrel{i}{\rightarrow}&1\\
&&j\downarrow&&\\
1&\stackrel{k}{\rightarrow}&1&&\\
l\downarrow&&&&\\
1&&&&\end{array}$$
$$\begin{array}{rcr}
1&&\\
i\uparrow&&\\
2&\stackrel{j}{\rightarrow}&1\\
k\downarrow&&\\
1&&\end{array}$$
In the first case, the indices $i,j,k,l\in\{1,\ldots,m\}$ fulfill $i\not=j\not=k\not=l$, and they are considered up to the symmetry $(i,j,k,l)\leftrightarrow(l,k,j,i)$. In the second case, the indices $i,j,k\in\{1,\ldots,m\}$ are pairwise different, and they are considered up to the natural $S_3$-action. We conclude that the fixed point set consists of
$$\frac{m(m-1)^3}{2}+\frac{m(m-1)(m-2)}{6}=\frac{m(m-1)(3m^2-5m+1)}{6}$$
isolated fixed points, so this number is precisely the Euler characteristic of the moduli space $M_{(2,3)}(K_m)$. This can also be obtained directly from Theorem \ref{betti}, but the advantage here is that we get a positive formula a priori (see the discussion following Theorem \ref{betti}).\\[1ex]
For more general coprime dimension vectors $d$, this approach is the central tool in \cite{Weist2} for obtaining exponential lower bounds for $\chi(M_d^{\rm st}(K_m))$, by constructing ``enough'' stable representations for the covering quiver.

\subsection{Cell decompositions and the Bialynicki-Birula method}\label{bb}

Suppose the rank $1$ torus ${\bf C}^*$ acts on a smooth projective variety $X$. The following is proved in \cite{BB1,BB2}:

\begin{theorem} Let $C$ be a connected component of the fixed point set $X^{{\bf C}^*}$, and let $A(C)$ be the set of all $x\in X$ such that $\lim_{t\rightarrow 0}xt\in C$. Then both $C$ and $A(C)$ are smooth, and $A(C)$ is locally closed in $X$. Associating to $x\in A(C)$ the limit $\lim_{t\rightarrow 0} xt$ defines a morphism $\pi:A(C)\rightarrow C$, which turns $A(C)$ into a Zariski locally trivial affine bundle. There exists a descending filtration by closed subvarieties $X=X_0\supset X_1\supset\ldots\supset X_t=\emptyset$ such that the successive complements $X_{i-1}\setminus X_i$ are precisely the $A(C)$.
\end{theorem}

In particular, in the case where $X^{{\bf C}^*}$ is finite, the sets $A(x)$ for $x\in X^{{\bf C}^*}$ yield a cell decomposition of $X$!\\[1ex]
We have seen in section \ref{lockronecker} that there are finitely many fixed points of $T_{K_m}$ acting on the moduli space $M_{(2,3)}^{\rm st}(K_m)$. Choosing a sufficiently general embedding of ${\bf C}^*$ into $T_{K_m}$, these are precisely the ${\bf C}^*$-fixed points, and the above theorem proves that $M_{(2,3)}^{\rm st}(K_m)$ admits a cell decomposition. See \cite{KW} for details. In the particular case $m=3$, the resulting cell decomposition has the following form:

\begin{proposition} If $X$ is a stable representation of the $3$-arrow Kronecker quiver of dimension vector $(2,3)$, then $X$ is isomorphic to exactly one of the following triples of $3\times 2$-matrices ($\ast$ indicating an arbitrary entry):
$$
\left[\begin{array}{cc} 1&~ \\ ~&1 \\ ~&~ \end{array}\right]
\left[\begin{array}{cc} ~&~ \\ 1&~ \\ ~&1 \end{array}\right]
\left[\begin{array}{cc} ~&~ \\ ~&~ \\ ~&~ \end{array}\right]
\mbox{ or }
\left[\begin{array}{cc} 1&~ \\ ~&~ \\ ~&1 \end{array}\right]
\left[\begin{array}{cc} ~&\ast \\ 1&~ \\ ~&~ \end{array}\right]
\left[\begin{array}{cc} ~&~ \\ ~&1 \\ ~&~ \end{array}\right]
\mbox{ or}
$$
\begin{sloppypar}
	
\end{sloppypar}
$$
\left[\begin{array}{cc} 1&~ \\ ~&\ast \\ ~&\ast \end{array}\right]
\left[\begin{array}{cc} ~&~ \\ 1&~ \\ ~&1 \end{array}\right]
\left[\begin{array}{cc} ~&~ \\ ~&1 \\ ~&~ \end{array}\right]
\mbox{ or }
\left[\begin{array}{cc} 1&~ \\ ~&1 \\ ~&\ast \end{array}\right]
\left[\begin{array}{cc} ~&~ \\ \ast&~ \\ ~&1 \end{array}\right]
\left[\begin{array}{cc} ~&~ \\ 1&~ \\ ~&~ \end{array}\right]
\mbox{ or}
$$
\begin{sloppypar}
	
\end{sloppypar}
$$
\left[\begin{array}{cc} 1&~ \\ ~&1 \\ ~&\ast \end{array}\right]
\left[\begin{array}{cc} ~&~ \\ 1&~ \\ ~&\ast \end{array}\right]
\left[\begin{array}{cc} ~&~ \\ ~&~ \\ ~&1 \end{array}\right]
\mbox{ or }
\left[\begin{array}{cc} 1&~ \\ ~&1 \\ ~&~ \end{array}\right]
\left[\begin{array}{cc} ~&~ \\ \ast&~ \\ \ast&\ast \end{array}\right]
\left[\begin{array}{cc} ~&~ \\ 1&~ \\ ~&1 \end{array}\right]
\mbox{ or}
$$
\begin{sloppypar}
	
\end{sloppypar}
$$
\left[\begin{array}{cc} \ast&\ast \\ 1&~ \\ ~&\ast \end{array}\right]
\left[\begin{array}{cc} 1&~ \\ ~&~ \\ ~&1 \end{array}\right]
\left[\begin{array}{cc} ~&~ \\ ~&1 \\ ~&~ \end{array}\right]
\mbox{ or }
\left[\begin{array}{cc} 1&~ \\ ~&\ast \\ ~&\ast \end{array}\right]
\left[\begin{array}{cc} ~&~ \\ 1&1 \\ ~&\ast \end{array}\right]
\left[\begin{array}{cc} ~&~ \\ ~&~ \\ ~&1 \end{array}\right]
\mbox{ or}
$$
\begin{sloppypar}
	
\end{sloppypar}
$$
\left[\begin{array}{cc} 1&~ \\ ~&\ast \\ ~&\ast \end{array}\right]
\left[\begin{array}{cc} ~&~ \\ ~&1 \\ \ast&\ast \end{array}\right]
\left[\begin{array}{cc} ~&~ \\ 1&~ \\ ~&1 \end{array}\right]
\mbox{ or }
\left[\begin{array}{cc} ~&\ast \\ 1&\ast \\ \ast&\ast \end{array}\right]
\left[\begin{array}{cc} 1&~ \\ ~&1 \\ ~&~ \end{array}\right]
\left[\begin{array}{cc} ~&~ \\ ~&~ \\ ~&1 \end{array}\right]
\mbox{ or}
$$
\begin{sloppypar}
	
\end{sloppypar}
$$
\left[\begin{array}{cc} \ast&~ \\ ~&1 \\ \ast&\ast \end{array}\right]
\left[\begin{array}{cc} 1&~ \\ ~&~ \\ ~&\ast \end{array}\right]
\left[\begin{array}{cc} ~&~ \\ 1&~ \\ ~&1 \end{array}\right]
\mbox{ or }
\left[\begin{array}{cc} \ast&~ \\ \ast&\ast \\ \ast&\ast \end{array}\right]
\left[\begin{array}{cc} 1&~ \\ ~&1 \\ ~&~ \end{array}\right]
\left[\begin{array}{cc} ~&~ \\ 1&~ \\ ~&1 \end{array}\right]
\mbox{ or}
$$
\begin{sloppypar}
	~
\end{sloppypar}
$$
\left[\begin{array}{cc} ~&\ast \\ 1&~ \\ \ast&\ast \end{array}\right]
\left[\begin{array}{cc} \ast&~ \\ ~&1 \\ \ast&\ast \end{array}\right]
\left[\begin{array}{cc} 1&~ \\ ~&~ \\ ~&1 \end{array}\right]
\mbox{    }
\left.\begin{array}{cc} ~&~ \\ ~&~ \\ ~&~ \end{array}\right.
\left.\begin{array}{cc} ~&~ \\ ~&~ \\ ~&~ \end{array}\right.
\left.\begin{array}{cc} ~&~ \\ ~&~ \\ ~&~ \end{array}\right. 
$$
\end{proposition}

\section{Arithmetic approach}\label{arithm}

\subsection{Cell decompositions and counting points over finite fields}\label{cdrp}

In the same spirit as the study of Betti numbers, counting rational points over finite fields can give predictions for the structure of potential cell decompositions. We present here the basic idea of this approach (without reference to schemes or other concepts from arithmetic algebraic geometry).\\[1ex]
Suppose $X$ is a quasiprojective variety, embedded as a locally closed subset of a projective space ${\bf P}^N({\bf C})$. Thus $X$ is given by certain polynomial equalities and inequalities in the coordinates, i.e.
$$X=\{(x_0:\ldots:x_N)\, :\, P_i(x_0,\ldots,x_N)=0,\, Q_j(x_0,\ldots,x_N)\not=0\}$$
for certain homogeneous polynomials $P_i$, $Q_j$. Suppose the coefficients appearing in these defining polynomials are all contained in some ring of algebraic numbers $R$ (thus a finite extension of ${\bf Z}$). For any prime $p\subset R$, the factor $R/p$ is a finite field $k$. We can then consider the defining conditions of $X$ modulo $p$ and thus get a locally closed subset $X(k)$ of ${\bf P}^n(k)$ (which is a finite set), and we can study its cardinality $|X(k)|$, and in particular how it depends on $k$.\\[1ex]
Suppose that $X$ is a variety admitting a cell decomposition $$X=X_0\supset X_1\supset\ldots\supset X_t=\emptyset\mbox{ with } X_{i-1}\setminus X_i\simeq{\bf A}^{d_i},$$
such that all steps $X_i$ of the filtration are also defined over $R$. Then $$|X(k)|=\sum_i(|X_{i-1}(k)|\setminus |X_i(k)|)=\sum_i|(X_{i-1}\setminus X_i)(k)|=\sum_i|{\bf A}^{d_i}(k)|=\sum_i|k|^{d_i}.$$
Thus, in this case, we can define a polynomial with nonnegative integer coefficients $$P_X(t)=\sum_iq^{d_i}\in{\bf N}[t]$$ such that $P_X(|k|)=|X(k)|$.\\[1ex]
Like in the setting of Betti numbers, we can get a prediction for the nature of a cell decomposition from this counting polynomial. Note also that the property of admitting such a counting polynomial is very special among all varieties. One standard example for varieties without counting polynomial are elliptic curves. Also note the following elementary example: consider a unit circle $X$, defined by $X=\{(x,y)\, :\, x^2+y^2=1\}$. It is defined over ${\bf Z}$, and using stereographic projection, it is a simple exercise to see that $|X(k)|$ equals $|k|$ or $|k|+1$, depending on whether $-1$ is a square in $k$ or not. On the other hand, the change of variables $u=x+iy$, $v=x-iy$ transforms $X$ to $\{(u,v)\, :\, uv=1\}$, a hyperbola with $|X(k)|=|k|-1$ for any finite field $k$. Thus, the number of points over finite fields, and in particular the property of admitting a counting polynomial, depends on the chosen embedding of $X$ into projective space.\\[1ex]
If $X$ admits a cell decomposition, we see from the above that $P_X(t)$ also equals $\sum_i\dim H^{2i}(X,{\bf Q})t^i$. This holds in much bigger generality for the class of so-called (cohomologically) pure varieties. Deligne's solution of the Weil conjectures \cite{De} states in particular that any smooth projective variety is pure. Thus, for a smooth projective variety admitting a counting polynomial, we automatically know the Betti numbers. This is the method with which Theorem \ref{betti} was obtained.\\[1ex]
For general quasi-projective varieties, one can at least compute the Euler characteristic from a counting polynomial as above, see e.g.~\cite{RCRP}:

\begin{lemma}\label{key} Suppose there exists a rational function $P_X(t)\in{\bf Q}(t)$ such that $P_X(|k|)=|X(k)|$ for almost all reductions $k$ of the ring $R$ over which $X$ is defined. Then $P_X(t)\in{\bf Z}[t]$ is actually a polynomial with integer coefficients, and $P_X(1)=\chi_c(X)$, the Euler characteristic of $X$ in singular cohomology with compact support.
\end{lemma}

We say that a variety $X$ (or, more precisely, a chosen model of $X$ over $R$) has the polynomial counting property if a polynomial as in the lemma exists.

\subsection{Counting points of quiver moduli over finite fields}\label{counting}

For quiver moduli, there is a canonical choice for a model over ${\bf Z}$, since we can define quiver representations over the integers, and since we have natural embeddings of quiver moduli into projective spaces by the results of section \ref{coord}. It is proved in \cite{RCRP} that quiver moduli fit into the above discussion nicely:

\begin{theorem}\label{simplecounting} For arbitrary quivers $Q$, stabilities $\Theta$ and dimension vectors $d$, both $M_d^{\rm st}(Q)$ and $M_d^{\rm sst}(Q)$ have the polynomial counting property. 
\end{theorem}

The proof uses Hall algebras in an essential way, see the sketch in section \ref{apphall}. In fact, in \cite{RCRP} only the case of $M_d^{\rm st}(Q)$ is considered, but the case of $M_d^{\rm sst}(Q)$ follows easily using the Luna stratification described in section \ref{bgp}.\\[1ex]
The corresponding counting polynomials $P_{M_d^{\rm st}(Q)}(q)$ and $P_{M_d^{\rm sst}(Q)}(q)$ count isomorphism classes of absolutely stable, resp. semistable, representations of $Q$ of dimension vector $d$ (here a representation is called absolutely stable if it remains stable under base extension to an algebraic closure of the finite field $k$ -- see Remark \ref{automhnf} for a short discussion of scalar extensions of (semi-)stable representations).\\[1ex]
To state an explicit formula for these polynomials, we have to introduce some notation. We consider the formal power series ring $F={\bf Q}(q)[[t_i\, :\, i\in I]]$ and define monomials $$t^d=\prod_{i\in I}t_i^{d_i}$$ for a dimension vector $d\in{\bf N}I$. Besides the usual commutative multiplication, we also consider the twisted multiplication $$t^d\circ t^e=q^{-\langle d,e\rangle}t^{d+e}$$ on $F$. Denote by $\psi_k$ for $k\geq 1$ the operator on $F$ defined by $$\psi_k(q)=q^k\mbox{ and }\psi_k(t^d)=t^{kd}.$$
We combine these operators into $$\Psi(f):=\sum_{k\geq 1}\frac{1}{k}\psi_k(f),$$ which by \cite{Card} has an inverse $$\Psi^{-1}(f)=\sum_{k\geq 1}\frac{\mu(k)}{k}\psi_k(f)$$ involving the number theoretic Moebius function $\mu(k)$. Bases on this, one defines mutually inverse operators $${\rm Exp}(f)=\exp(\Psi(f))\mbox{ and }{\rm Log}(f)=\Psi^{-1}(\log(f))$$ on $F$. Using these concepts, an explicit formula for the counting polynomials is proved in \cite{MR}:

\begin{theorem}\label{formulanumbstab} For any $Q$, $\Theta$, $d$ as above and any $\mu\in{\bf Q}$, we have the following formulas which make key use of the rational functions $P_d(q)$ introduced in Definition \ref{defpdq}:
$$(\sum_{d\in{\bf N}I_\mu}P_d(q)t^d)\circ{\rm Exp}(\frac{1}{1-q}\sum_{d\in{\bf N}I_\mu}P_{M_d^{\rm st}(Q)}(q)t^d)=1$$
and
$$\sum_{d\in{\bf N}I_\mu}P_{M_d^{\rm sst}(Q)}(q)t^d={\rm Exp}(\sum_{d\in{\bf N}I_\mu}P_{M_d^{\rm st}(Q)}(q)t^d).$$
\end{theorem}

Let us consider the special case $\Theta=0$. Then the first formula of the above theorem can be made more explicit:

\begin{corollary} We have
$$(\sum_{d\in{\bf N}I}q^{-\langle d,d\rangle}\prod_{i\in I}\prod_{k=1}^{d_i}(1-q^{-k})^{-1}t^d)\circ{\rm Exp}(\frac{1}{1-q}\sum_{d\in{\bf N}I}P_{M_d^{\rm simp}(Q)}(q)t^d)=1.$$
\end{corollary}

Here are some examples of the counting polynomials in the special case of the $m$-loop quiver $L_m$, denoting $a_d(q)=P_{M_d^{\rm simp}(L_m)}(q)$:
$$a_1(q)=q^m,$$
$$a_2(q)=q^{2m}\frac{(q^m-1)(q^{m-1}-1)}{q^2-1},$$
$$a_3(q)=q^{3m+1}\frac{(q^m-1)(q^{2m-2}-1)(q^{2m-2}(q^m+1)-q^{m-2}(q+1)^2+q+1)}{(q^3-1)(q^2-1)}.$$

Based on computer experiments for the $m$-loop quiver and dimensions up to $12$, the following conjecture is made in \cite{MR}:

\begin{conjecture}\label{conjpos} When written as a polynomial in the variable $q-1$, all polynomials $P_{M_d^{\rm simp}(Q)}(q)$ have nonnegative coefficients, i.e.~$P_{M_d^{\rm simp}(Q)}(q)\in{\bf N}[q-1]$.
\end{conjecture}

Even more optimistically, this conjecture suggests that there should be a decomposition into tori of any such moduli space, compare the discussion in section \ref{ned}.

\subsection{Moduli of simple representations}\label{simples}

It turns out that the first two terms in the Taylor expansion around $q=1$ can be computed. The constant term is zero except in case $d=e_i$ for some $i\in I$: the action of the torus ${\bf C}^*$, diagonally embedded into $T_Q$, has no fixed points. We can factor by this action to obtain a projectivization ${\bf P}M_d^{\rm simp}(Q)$ of $M_d^{\rm simp}(Q)$ (but note that ${\bf P}M_d^{\rm simp}(Q)$ is not a projective variety since $M_d^{\rm simp}(Q)$ is not necessarily affine).\\[1ex]
To state a formula for the linear term, we need some notation. We consider oriented cycles in the quiver $Q$, written as $\omega=\alpha_s\ldots\alpha_1$ for arrows $\alpha_1,\ldots,\alpha_s$ in $Q$. We have the notion of cyclic equivalence of cycles as the equivalence relation generated by $$\alpha_1\ldots\alpha_s\sim\alpha_2\ldots\alpha_s\alpha_1.$$
We call a cycle $\omega$ primitive if it is not cyclically equivalent to a proper power of another cycle, i.e.~$\omega\not\sim(\omega')^n$ for all $n>1$. A notion of dimension vector $\dimv\omega$ of a cycle $\omega$ can be defined by setting $(\dimv\omega)_i$ to equal the number of times $\omega$ passes through the vertex $i\in I$. Then the following formula is proved in \cite{RLoc}:

\begin{theorem} The Euler characteristic (in cohomology with compact support) of ${\bf P}M_d^{\rm simp}(Q)$, or, equivalently, the constant term in the Taylor expansion of $P_{M_d^{\rm simp}(Q)}(q)$ around $q=1$, is given as the number of cyclic equivalence classes of primitive cycles in $Q$ of dimension vector $d$.
\end{theorem}

The proof of this theorem works via the localization principle, in a refined form so that it applies to the action of $T_Q$ on ${\bf P}M_d^{\rm simp}(Q)$. The additional difficulty is that $PG_d$ does not act freely on the projectivization of the subset $R_d^{\rm simp}(Q)$ of $R_d(Q)$ consisting of the simple representation, so that the argument of section \ref{locqm} does not apply. Instead, this refined form of localization leads to the following description of the fixed point components:\\[1ex]
For any indivisible element $\nu\in{\bf N}Q_1$ (i.e.~${\rm gcd}(\nu_\alpha\, :\, \alpha\in Q_1)=1$), define a covering quiver $\widehat{Q}_\nu$ of $Q$ as follows: the vertices are given by $$(\widehat{Q}_\nu)_0=Q_0\times{\bf Z}Q_1/{\bf Z}\nu,$$
and the arrows are given as $$(\widehat{Q}_\nu)_1=\{(\alpha,\overline{\lambda}):(i,\overline{\lambda})\rightarrow (j,\overline{\lambda+\alpha})\,:\, (\alpha:i\rightarrow j)\in Q_1,\, \overline{\lambda}\in{\bf Z}Q_1/{\bf Z}\nu\}.$$
The group ${\bf Z}Q_1$ acts on $\widehat{Q}_\nu$ by translation, and we consider dimension vectors $\widehat{d}$ for $\widehat{Q}_\nu$ up to translational equivalence. Each $\widehat{d}$ projects to a dimension vector for $Q$.

\begin{theorem} The fixed point set ${\bf P}M_d^{\rm simp}(Q)^{T_Q}$ is isomorphic to the disjoint union of moduli spaces of the same type ${\bf P}M_{\widehat{d}}^{\rm simp}(\widehat{Q}_\nu)$, where $\nu$ ranges over the indivisible elements of ${\bf Z}Q_1$, and $\widehat{d}$ ranges over the equivalence classes of dimension vectors for $\widehat{Q}_\nu$ projecting to $d$.
\end{theorem}

Since the Euler characteristic (in singular cohomology with compact support) is invariant under taking torus fixed points, and is additive with respect to disjoint unions, we can apply this theorem repeatedly. It turns out that, after finitely many localizations, the only resulting covering quivers contributing with non-zero Euler characteristic are cyclic quivers with dimension vector equal to $1$ at any vertex; they contribute with an Euler characteristic equal to $1$ since the projectivized space of simples is just a single point. These covering quivers are in fact parametrized by the cyclic equivalence classes of primitive cycles, yielding the claimed formula.\\[1ex]
Although the Euler characteristic is computed in the proof as a purely combinatorial number, it admits various algebraic interpretations, most notably in relation to the Hochschild homology of the path algebra $kQ$. Namely, the zero-th Hochschild homology $HH_0(kQ)$ equals $kQ/[kQ,kQ]$, the path algebra modulo additive commutators, an object which plays a central role in some approaches to noncommutative algebraic geometry (see e.g. \cite{Ginz}). Now $HH_0(kQ)$ inherits a ${\bf N}I$-grading from $kQ$, and the degree $d$-part has a basis consisting of cyclic equivalence classes of cycles of dimension vector $d$. A similar result holds for $HH_1(kQ)$ (all other Hochschild homology being zero).\\[1ex]
The Euler characteristic also admits a representation-theoretic interpretation: suppose a primitive cycle
$$\omega\, :\, i_0\stackrel{\alpha_1}{\rightarrow}i_1\stackrel{\alpha_2}{\rightarrow}\ldots\stackrel{\alpha_s}{\rightarrow}
i_s=i_0$$ of dimension vector $d$ in $Q$ is given. For each vertex $i\in Q_0$, define
$${\cal K}_i=\{k=0,\ldots,s-1\, :\, i_k=i\}.$$
Consider the $d_i$-dimensional vector space $M_i$ with basis elements $b_k$ for $k\in {\cal K}_i$.
For each arrow $(\alpha:i\rightarrow j)\in Q_1$ and each $k\in {\cal K}_i$, define
$$M(\omega)_\alpha(b_k)=\left\{\begin{array}{ccc}b_{k+1}&,&\alpha=\alpha_{k+1},\\ 0&,&\mbox{otherwise}.
\end{array}\right\}$$
This defines a representation $M(\omega)$. It is easy to see that $M(\omega)\simeq M(\omega')$ if and only if $\omega$ and $\omega'$ are cyclically equivalent, and that $M(\omega)$ is simple if and only if $\omega$ is primitive. One can hope that these representations enter as the $0$-dimensional strata of a conjectural decomposition of ${\bf P}M_d^{\rm simp}(Q)$ into tori, see the end of section \ref{counting}.\\[1ex]
It would be very interesting to have a description of all the individual Betti numbers in singular cohomology with compact support of ${\bf P}M_d^{\rm simp}(Q)$, and to see how they relate to e.g.~Hochschild homology. But such a description cannot be obtained in an obvious way via localization, since the projectivized space of simples is not projective. Again (compare the discussion in section \ref{ned}), a smooth compactification is missing.\\[1ex]
The knowledge so far about the counting polynomials $P_{M_d^{\rm simp}(Q)}(q)$ suggests to view them in analogy to the polynomials $i_d(q)$ counting isomorphism classes of absolutely indecomposable representations of a quiver without oriented cycles, which are the basis for the Kac conjectures \cite{Kac15}.  These state that $i_d(0)$ equals the multiplicity of the root $d$ in the root system associated to $Q$ (compare section \ref{kt}), and that $i_d(q)\in{\bf N}[q]$. The first conjecture is proved in the case of indivisible $d$ in \cite{CBVdB}; a proof for arbitrary $d$ is announced in \cite{Hausel}.\\[1ex]
Compare this to the results and conjectures above: we have a known number $\frac{P_{M_d^{\rm simp}(Q)}(q)}{q-1}|_{q=1}$ with combinatorial and algebraic interpretations, and we conjecture that $P_{M_d^{\rm simp}(Q)}(q)\in{\bf N}[q-1]$.\\[1ex]
This suggests some deep similarities between the counting of indecomposables and the counting of simples.\\[1ex]
It would also be very interesting to have a better understanding of the Euler characteristics $\chi_c(M_d^{\rm st}(Q))$ for non-trivial $\Theta$ and non-coprime $d$. In principle, the above formulas allow to determine this number by evaluation of the counting polynomials at $q=1$, but again, more explicit (combinatorial) formulas are desirable.

\section{The role of Hall algebras}\label{rolehall}

Several of the theorems above on Betti numbers and numbers of points over finite fields of quiver moduli can be proved using calculations in the Hall algebra of a quiver as introduced in \cite{RiHall}.

\subsection{Definition of Hall algebras}

Let $k$ be a finite field with $q$ elements. Let $H_q(Q)$ be a ${\bf Q}$-vector space with basis elements $[M]$ indexed by the isomorphism classes of $k$-representations of $Q$. Define a multiplication on $H_q(Q)$ by
$$[M]\cdot[N]:=\sum_{[X]}F_{M,N}^X\cdot[X],$$
where $F_{M,N}^X$ denotes the number of subrepresentations $U$ of $X$ which are isomorphic to $N$, with quotient $X/U$ isomorphic to $M$. This number is obviously finite. Also note that the sum in the definition of the multiplication is finite, since $F_{M,N}^X\not=0$ implies $\dimv X=\dimv M+\dimv N$, and there are only finitely many (isomorphism classes of) representations of fixed dimension vector.\\[1ex]
The above multiplication defines an associative ${\bf Q}$-algebra structure on $H_q(Q)$ with unit $1=[0]$. This algebra is naturally ${\bf N}I$-graded by the dimension vector.\\[1ex]
We will also consider a completed (with respect to the maximal ideal spanned by non-zero representations) version of the Hall algebra, thus $$H_q((Q))=\prod_{[M]}{\bf Q}[M],$$
with the same multiplication as before. This version has the advantage that certain ``generating series'', like e.g. $\sum_{[M]}[M]$, can be considered in it.

\subsection{Hall algebras and quantum groups}

The Hall algebra is usually considered in relation to quantum groups: let $Q$ be a quiver without oriented cycles, and define $C_q(Q)$ (the composition algebra) as the subalgebra of $H_q(Q)$ generated by the basis elements $[S_i]$ corresponding to the simple representations $S_i$ for $i\in I$. Let $C$ be the matrix representing the symmetric bilinear form $(,)$, to which we can associate a Kac-Moody algebra $\mathfrak{g}$, see \cite{Kac2}. Its enveloping algebra ${\cal U}(\mathfrak{g})$ admits a quantum deformation, the quantized enveloping algebra ${\cal U}_q(\mathfrak{g})$ \cite{Jantzen, Lusztig}. We will only consider its positive part ${\cal U}_q^+(\mathfrak{g})$ (induced from the triangular decomposition of the Lie algebra $\mathfrak{g}$), which can be defined as the ${\bf Q}(q)$-algebra with (Chevalley) generators $E_i$ for $i\in I$ and defining relations (the $q$-Serre relations)
$$\sum_{k+l=1-c_{ij}}\left[{{1-c_{ij}}\atop k}\right]E_i^kE_jE_i^l=0\mbox{ for all }i\not=j\mbox{ in }I.$$

There is a twisted version of the Hall algebra, which we denote by $H_q(Q)^{\rm tw}$; it is defined in the same way as $H_q(Q)$, but the multiplication is twisted by a power of $q$, namely
$$[M]\cdot[N]=q^{-\langle \dimv M,\dimv N\rangle}\sum_{[X]}F_{M,N}^X[X].$$
The following is proved in \cite{Green}:

\begin{theorem} The composition subalgebra $C_q(Q)^{\rm tw}$ of $H_q(Q)^{\rm tw}$ is isomorphic to the specialization of ${\cal U}_q^+(\mathfrak{g})$ at $q=|k|$.
\end{theorem}


\subsection{Applications of Hall algebras to quiver moduli}\label{apphall}

The Hall algebra admits an evaluation homomorphism to a skew polynomial ring: as in section \ref{counting}, consider the ring ${\bf Q}_q[I]$ which has basis elements $t^d$ for $d\in{\bf N}I$ and multiplication $$t^d\cdot t^e=q^{-\langle d,e\rangle}t^{d+e}.$$
We have a natural skew formal power series version ${\bf Q}_q[[I]]$ of ${\bf Q}_q[I]$. Then we can define the evaluation morphism as in \cite{RCRP}:

\begin{lemma} The map sending $[M]$ to $\frac{1}{|{\rm Aut}(M)|}\cdot t^{\dimv M}$ induces ${\bf Q}$-algebra morphisms $\int:H_q(Q)\rightarrow {\bf Q}_q[I]$ and $\int:H_q((Q))\rightarrow {\bf Q}_q[[I]]$, respectively.
\end{lemma}

We will now consider some special elements in $H_q((Q))$ and their evaluations.\\[1ex]
Consider $$e_d=\sum_{\dimv M=d}[M],$$
the sum over all isomorphim classes of representations of dimension vector $d$. We have
$$\int e_d=\sum_{\dimv M=d}\frac{1}{{\rm Aut}(M)}t^d=\sum_{\dimv M=d}\frac{|G_dM|}{|G_d|}t^d=\frac{|R_d(Q)|}{|G_d|}t^d,$$
since the cardinality of the orbit $G_dM$ of $M$ in $R_d$ equals the order of the group $G_d$, divided by the order of the stabilizer, which by definition equals the automorphism group of $M$.\\[1ex]
Next, consider $$e_d^{\rm sst}=\sum_{\substack{{\dimv M=d}\\ {M\text{ semistable}}}}[M],$$
the sum over all isomorphism classes of semistable representations of dimension vector $d$. Similarly to the above, we have $$\int e_d^{\rm sst}=\frac{|R_d^{\rm sst}(Q)|}{|G_d|}.$$

By the results of section \ref{aas} on the Harder-Narasimhan filtration, every representation $M$ admits a unique Harder-Narasimhan filtration
$$0=M_0\subset M_1\subset\ldots\subset M_s=M.$$
Let $d_i$ be the dimension vector of the subquotient $M_i/M_{i-1}$, for $i=1\ldots s$. All the subquotients being semistable, and the HN-filtration being unique, we see that $[M]$ appears with coefficient equal to $1$ in the product $$e_{d_s}^{\rm sst}\cdot\ldots\cdot e_{d_2}^{\rm sst}\cdot e_{d_1}^{\rm sst}.$$
Existence of the HN filtration yields the following identity:

\begin{lemma} We have $$e_d=\sum_*e_{d_s}^{\rm sst}\cdot\ldots e_{d_1}^{\rm sst},$$
the sum running over all decompositions $d_1+\ldots+d_s=d$ of $d$ into non-zero dimension vectors such that $\mu(d_1)>\ldots>\mu(d_s)$.
\end{lemma}

We can thus determine any $e_d^{\rm sst}$ inductively, the induction starting at dimension vectors for which every representation is semisimple (this is equivalent to $\Theta$ being constant on the support of $d$). Applying the evaluation map $\int$, this gives

\begin{corollary} We have
$$\frac{|R_d^{\rm sst}(Q)|}{|G_d|}=\frac{|R_d(Q)|}{|G_d|}-\sum_*q^{-\sum_{k<l}\langle d^l,d^k\rangle}\prod_{k=1}^s\frac{|R_{d^k}^{\rm sst}(Q)|}{|G_{d^k}|},$$
the sum running over all decompositions as above of length $s\geq 2$.
\end{corollary}

General arithmetic considerations prove that, in case $d$ is $\Theta$-coprime, we have $$|M_d^{\rm sst}(Q)|=(q-1)\frac{|R_d^{\rm sst}(Q)|}{|G_d|}$$
(essentially since $M_d^{\rm sst}(Q)$ is the quotient of $R_d^{\rm sst}(Q)$ by the group $PG_d$, which acts freely in the coprime case). This gives an explicit (recursive) formula for the number of rational points of $M_d^{\rm sst}(Q)$. The result Theorem \ref{betti} is obtained from this by applying Deligne's solution of the Weil conjectures (see section \ref{cdrp}) to pass from points over finite fields to Betti numbers, and by giving an explicit resolution of the recursion.\\[1ex] 
As a second example of the use of Hall algebras, used in \cite{RCRP}, consider the element $e_\mu=\sum_{d\in{\bf N}I_\mu}e_d^{\rm sst}\in H_q((Q))$.

\begin{lemma} We have $$e_\mu^{-1}=\sum_{[M]}\gamma_M[M],$$
where $\gamma_M$ is zero if $M$ is not polystable, and $$\gamma_M=\prod_{[S]}(-1)^{m_S}|{\rm End}(S)|^{{m_S}\choose 2}$$ if $M=\bigoplus_SS^{m_S}$ as a direct sum of stable representations of slope $\mu$.
\end{lemma}

What is surprising about this lemma is that the inverse is a sum over polystable representations only. Applying the evaluation map, we get

\begin{corollary}\label{basisfornumb} We have $$\sum_{d\in{\bf N}I_\mu}\frac{|R_d^{\rm sst}(Q)|}{|G_d|}=\sum_{m:\mathcal{S}\rightarrow{\bf N}}\prod_{[S]}(-1)^{m_S}\frac{{|{\rm End}(S)|}^{{m_S}\choose 2}}{|{\rm  Aut}(S^{m_S})|}t^{\sum_Sm_S\dimv S},$$
where the sum runs over all maps (with finite support) from the set $\mathcal{S}$ of isomorphism classes of stable representations of slope $\mu$ to the set of nonnegative integers.
\end{corollary}

This identity forms the basis for the proof of Theorems \ref{simplecounting} and \ref{formulanumbstab}: roughly, we can single out the summands corresponding to stable representations (i.e.~the function $m$ has precisely one non-zero value, equal to $1$) to get a recursive formula for the number of isomorphism classes of stables in terms of the rational functions $P_d(q)$. Some arithmetic considerations allow passage to absolutely stable representations. Using Lemma \ref{key}, the theorem follows.\\[1ex]
As a third application, we prove the cohomology formula Theorem \ref{cohomsm} for the smooth models of section \ref{sm} using Hall algebras. For some $n\in{\bf N}I$, consider the finitely generated projective representation $P^{(n)}=\bigoplus_{i\in I}P_i^{n_i}$. Besides the element $e_\mu\in H_q((Q))$, consider the following elements:
$$h_{\mu,n}=\sum_{[M]\in{\rm mod}_\mu kQ}|{\rm Hom}(P^{(n)},M)|[M],$$
$$e_{\mu,n}=\sum_{[M]\in{\rm mod}_\mu kQ}|{\rm Hom}^0(P^{(n)},M)|[M],$$
where ${\rm Hom}^0(Z,M)$ denotes the set of all maps $f:Z\rightarrow M$ with the following property: if $f(Z)\subset U\subset M$ for $U\in{\rm mod}_\mu(Q)$, then $U=M$. Then the following identity holds, see \cite{ERSM}:
\begin{lemma}\label{app3} We have $e_\mu h_{\mu,n}=e_{\mu,n}$ in $H_q((Q))$.
\end{lemma}

Application of the evaluation map immediately yields the formula of Theorem \ref{cohomsm}.

\section{Smooth models and Hilbert schemes}\label{sm}

In this final section, we will consider a variant of quiver moduli (in some respect analogous to the quiver varieties of \cite{Nak2}) which enjoys several of the desirable properties which the original moduli lack in general: they admit universal bundles, they are always smooth and projective (over the moduli of semisimple representations), their Betti numbers can be calculated and the Poincare polynomial equals the counting polynomial for points over finite fields. In special cases we can even construct a cell decomposition and thus give a normal form. The drawback is that these moduli do not parametrize just isomorphism classes of representations, but equivalence classes of representations together with an additional structure. The material of this section is contained in \cite{ERSM}.

\subsection{Definition of smooth models}\label{dsm}

Let $Q,d,\Theta$ be a quiver, a dimension vector and a stability as before. Choose another dimension vector $d\in{\bf N}I$ and consider extended quiver data $\widetilde{Q},\widetilde{d},\widetilde{\Theta}$ defined as follows: the vertices of $\widetilde{Q}$ are given by $\widetilde{Q}_0=Q_0\cup\{\infty\}$, and the arrows of $\widetilde{Q}$ are those of $Q$, together with $n_i$ arrows from the additional vertex $\infty$ to any $i\in I$. We extend $d$ to a dimension vector $\widetilde{d}$ for $\widetilde{Q}$ by defining $\widetilde{d}_\infty=1$, and we define a stability $\widetilde{\Theta}$ for $\widetilde{Q}$ by setting $\widetilde{\Theta}_i=N\Theta_i$ for $i\in I$ and $\widetilde{\Theta}_\infty=\Theta(d)+1$, for some sufficiently large integer $N\in{\bf N}$.\\[1ex]
It is now easy to see that (because of the additional entry $1$ in $\widetilde{d}$) the dimension vector $\widetilde{d}$ is always $\widetilde{\Theta}$-coprime, so that the resulting moduli space of (semi-)stable representation $M_{\widetilde{d}}(\widetilde{Q})$ is smooth. We denote this moduli space by $M_{d,n}^\Theta(Q)$.

\begin{theorem} The moduli space $M_{d,n}^{\Theta}(Q)$ is smooth, and projective over the moduli of semisimple $M_d^{\rm ssimp}(Q)$. It parametrizes pairs consisting of a semistable representation $M$ of $Q$ of dimension vector $d$, together with a map $f:P^{(n)}\rightarrow M$ from the finitely generated projective representation $$P^{(n)}=\bigoplus_{i\in I}P_i^{n_i}$$ to $M$ with the property:
\begin{center} if $U$ is a proper subrepresentation of $M$ containing ${\rm Im}(f)$, then $\mu(U)<\mu(M)$.\end{center}
These pairs are parametrized up to isomorphisms respecting the additional data, i.e.~$(M,f)$ and $(M',f')$ are equivalent if there exists an isomorphism $\varphi:M\rightarrow M'$ such that $f'=\varphi\circ f$.
\end{theorem}

In other words, this moduli space, which we call a smooth model (for $M_d^{\rm sst}(Q)$), parametrizes semistable representations $M$ together with a map from a fixed projective to $M$ which ``avoids all subrepresentations contradicting stability of $M$''.\\[1ex]
The map forgetting the extra datum of the map $f$ induces a projective morphism $$\pi:M_{d,n}^\Theta(Q)\rightarrow M_d^{\rm sst}(Q).$$ The fibres of this map can be described explicitely using the Luna stratification of section \ref{bgp}. In particular, the generic fibre - the fibre over the stable locus $M_d^{\rm st}(Q)$ -- is isomorphic to projective space ${\bf P}^{n\cdot d-1}$, where $n\cdot d=\sum_{i\in I}n_id_i$.\\[1ex]
In the case where $d$ is $\Theta$-coprime, the smooth model stays very close to the original moduli space: it is isomorphic to the projectivization of the bundle $\bigoplus_{i\in I}\mathcal{V}_i^{n_i}.$
In all other cases, the smooth models $M_{d,i}^\Theta(Q)$ can therefore be viewed as ``projectivizations of non-existing universal bundles''. 

\subsection{Cohomology of smooth models}

Since the smooth models are a particular case of moduli in the coprime case, we know from the discussion in section \ref{sectbetti} that the odd Betti numbers vanish, and that the even Betti numbers can be computed using Theorem \ref{betti}. We will make this formula more explicit, using again the rational functions $P_d(q)$ of Definition \ref{defpdq}.

\begin{theorem}\label{cohomsm} In the skew formal power series ring ${\bf Q}_q[[I]]$, we have the following identity:
$$\sum_{d\in{\bf N}I_\mu}\sum_i\dim H^i((M_{d,n}^\Theta(Q),{\bf Q})q^{i/2}t^d=(\sum_{d\in{\bf N}I_\mu}P_d(q)t^d)^{-1}\cdot(\sum_{d\in{\bf N}I_\mu}q^{n\cdot d}P_d(q)t^d).$$
\end{theorem}

As noted in section \ref{apphall}, the result uses the identity of Lemma \ref{app3} in the Hall algebra of $Q$ and the passage from counting points over finite fields to Betti numbers, as in section \ref{cdrp}.

\subsection{Hilbert schemes}\label{hilbert}

We consider the special case $\Theta=0$ and denote the smooth model $M_{d,n}^0(Q)$ by ${\rm Hilb}_{d,n}(Q)$, which we call a Hilbert scheme for $Q$. It parametrizes (arbitrary) representations $M$ of dimension vector $d$, together with a surjective map from $P^{(n)}$ to $M$. In other words, it parametrizes subrepresentations $U\subset P^{(n)}$ of finitely generated projective representations such that $\dimv P^{(n)}/U=d$.\\[1ex]
In this case, there are much more explicit results on the structure of that space. For example, in the case of quivers without oriented cycles, ${\rm Hilb}_{d,n}(Q)$ can be described as an iterated Grassmann bundle, see \cite{RFQM}.\\[1ex]
For general quivers, we can give an explicit (and non-recursive) criterion for non-emptyness of ${\rm Hilb}_{d,n}(Q)$ (compare the discussion of non-emptyness in section \ref{stableschur}):

\begin{theorem} Let $Q$ be an arbitrary quiver, and let $d$ and $n$ be dimension vectors. We have ${\rm Hilb}_{d,n}(Q)\not=\emptyset$ if and only if the following two conditions are fulfilled:
\begin{enumerate}
\item $n_i\geq\langle d,i\rangle$ for all $i\in I$,
\item for every $i\in{\rm supp}(d)$ there exists $j\in{\rm supp}(n)$ and a path from $j$ to $i$ in ${\rm supp}(d)$.
\end{enumerate}
\end{theorem}

We also have a positive combinatorial formula for the Betti numbers (compare the discussion following Theorem \ref{betti}), based on certain multipartitions:\\[1ex]
Let $\Lambda_d$ be the set of tuples $(\lambda^i=(\lambda^i_1\geq\lambda^i_2\geq\ldots\geq\lambda^i_{d_i}\geq 0))_{i\in I}$ of partitions of length $d_i$, which we call multipartitions of length $d$. The weight of such a multipartition $\lambda$ is defined as $$|\lambda|=\sum_{i\in I}\sum_{k=1}^{d_i}\lambda^i_k.$$
We define a subset $S_{d,n}$ consisting of multipartitions $\lambda$ of length $d$ as fulfilling the following condition:
\begin{center}for every $0\leq e<d$, there exists a vertex $i\in I$ such that $\lambda^i_{d_i-e_i}<n_i-\langle e,i\rangle$.\end{center}

\begin{theorem} We have the following formula for the Betti numbers of the Hilbert scheme ${\rm Hilb}_{d,n}(Q)$:
$$\sum_i\dim H^i({\rm Hilb}_{d,n}(Q),{\bf Q})q^{i/2}=q^{n\cdot d-\langle d,d\rangle}\sum_{\lambda\in S_{d,n}}q^{-|\lambda|}.$$
\end{theorem}

\subsection{Cell decompositions for Hilbert schemes}\label{cellhilb}

In fact, all ${\rm Hilb}_{d,n}(Q)$ admit cell decompositions, which we will now construct (the special case of the multiple loop quivers was considered before in \cite{RCNH}). We need some combinatorial notation.\\[1ex]
For each vertex $i\in I$, define a covering quiver (in fact, a tree) $Q_i$ of $Q$ as follows: the vertices of $Q_i$ are parametrized by the paths $\omega$ on $Q$ starting in $i$. The arrows in $Q_i$ are given by $\alpha:\omega\rightarrow(\alpha\omega)$ for arrows $(\alpha:j\rightarrow k)\in Q_1$ and paths $\omega$ in $Q$ starting in $i$ and ending in $j$. Obviously $Q_i$ has a unique source corresponding to the empty path at $i$. We have a natural projection from $Q_i$ to $Q$ associating to a vertex $\omega$ of $Q_i$ the terminal vertex of the path $\omega$. A full subquiver $T$ of $Q_i$ is called a tree if it is closed under predecessors. To such a tree $T$, we can associate a dimension vector $d(T)$ for $Q$, where $d(T)_j$ is defined as the number of vertices in $T$ whose corresponding paths $\omega$ have terminal vertex $j$.\\[1ex]
For $n\in{\bf N}I$, we define $Q_n$ as the disjoint union of $n_i$ copies of each $Q_i$, for $i\in I$. The vertices of $Q_n$ are labelled by triples $(i,j,\omega)$, which means that the path $\omega$ starting in $i$ is placed in the $j$-th copy of $T_i$. An $n$-forest $T_*$ is a full subquiver of $Q_n$ which is closed under predecessors; in other words, it is a tuple $(T_{ij})_{i\in I,\, k=1,\ldots,n_i}$ of trees $T_{ij}$ in $Q_i$. The dimension vector $d(T_*)$ of an $n$-forest $T_*$ is defined as $\sum_{i,j}d(T_{ij})$.

\begin{theorem} Each ${\rm Hilb}_{d,n}(Q)$ admits a cell decomposition, whose cells are parametrized by $n$-forests of dimension vector $d$.
\end{theorem}

To construct the cells, we need a total ordering on $n$-forests. First, choose an arbitary total ordering on $I$. For each pair of vertices $i,j\in I$, choose a total ordering on the arrows from $i$ to $j$. Then, define a total ordering on $Q_1$ as follows: $(\alpha:i\rightarrow j)\leq(\beta:k\rightarrow l)$ if one of the following conditions holds:
\begin{itemize}
\item $i=k$, $j=l$ and $\alpha\leq\beta$ in the ordering chosen on the arrows from $i$ to $j$,
\item $i=k$ and $j<l$,
\item $i<k$.
\end{itemize}
Now we define a total ordering on the vertices of $Q_1$: let $\omega=(\alpha_s\ldots\alpha_1)$ and $\omega'=(\beta_t\ldots\beta_1)$ be two paths in $Q$ starting in $i$. Then $\omega\leq\omega'$ if $\alpha_k<\beta_k$ for the minimal index $k$ such that $\alpha_k\not=\beta_k$; if no such $k$ exists, we set $\omega\leq\omega'$ if $s\leq t$.\\[1ex]
Finally, we define a total ordering on the vertices of $Q_n$: we define $(i,j,\omega)\leq (i',j',\omega')$ if one of the following conditions holds:
\begin{itemize}
\item $i<i'$ in the total ordering on $I$,
\item $i=i'$ and $j<j'$,
\item $i=i'$, $j=j'$ and $\omega\leq\omega'$ in the total ordering on vertices of $Q_i$.
\end{itemize}

Define the corona $C(T_*)$ of $T_*$ as the set of all vertices $(i,j,\omega)$ of $Q_n$ which are not elements of $T_*$, but whose (unique) immediate predeccessor in $Q_n$ is an element of $T_*$.\\[1ex]
Choose a basis $v_{ij}$ for each vector space $V_i$. For a representation $M$ and a path $\omega=\alpha_s\ldots\alpha_1$ in $Q$, write $M_\omega=M_{\alpha_s}\circ\ldots\circ M_{\alpha_1}$. Given an $n$-forest $T_*$ of dimension vector $d$, let $Z_{T_*}$ be the set of all points $(M,f)$ such that the following conditions hold:
\begin{itemize}
\item the collection of elements $b_{(i,j,\omega)}=M_\omega(f_i(v_{i,j}))$ for $(i,j,\omega)\in T_*$ (i.e.~$\omega\in T_{i,j}$ for all $i\in I$ and $j=1,\ldots,n_i$) forms a basis of $M=\bigoplus_{i\in I}M_i$,
\item for each $(i,j,\omega)\in C(T_*)$, the element $M_\omega(f_i(v_{i,j}))$ belongs to the span of all suitable $b_{(i',j',\omega')}$ for $(i',j',\omega')\in T_*$ such that $(i',j',\omega')<(i,j,\omega)$.
\end{itemize}

\begin{theorem} There exists a filtration ${\rm Hilb}_{d,n}(Q)=X_0\supset X_1\supset\ldots\supset X_t=\emptyset$ such that the successive complements $X_{q-1}\setminus X_{q}$ are precisely the sets $Z_{T_*}$ defined above, where $T_*$ runs over all $n$-forests of dimension vector $d$. Consequently, ${\rm Hilb}_{d,n}(Q)$ admits a cell decomposition, whose cells are parametrized by the $n$-forests of dimension vector $d$.
\end{theorem}

As an immediate corollary, we can describe $\chi({\rm Hilb}_{d,n}(Q))$ as the number of $n$-forests of dimension vector $d$. Combinatorial considerations allow to describe the generating functions of Euler characteristics as solutions to algebraic equations in the formal power series ring ${\bf Q}[[I]]$ as follows:

\begin{corollary} For $n\in{\bf N}I$, define $$F_n(t)=\sum_{d\in{\bf N}I}\chi_c({\rm Hilb}_{d,n}(Q))t^d.$$ Then we have
$$F_n(t)=\prod_{i\in I}F_i(t)^{n_i}$$
and
$$F_i(t)=1+t_i\cdot\prod_{\alpha:i\rightarrow j}F_j(t)\mbox{ for all }i\in I.$$
\end{corollary}

One might conjecture that the generating functions $\sum_{d\in{\bf N}I_\mu}\chi_c(M_{d,n}^\Theta(Q))t^d\in{\bf Q}[[I]]$ for arbitrary smooth models are always algebraic.

\end{document}